\documentclass[a4paper,10pt]{article}
\usepackage[utf8]{inputenc}

\usepackage{amsmath}
\usepackage{graphicx}
\usepackage[caption = false]{subfig}
\usepackage{float}
\usepackage{amsthm}
\usepackage{pdfsync}
\usepackage{wasysym}
\usepackage[english]{babel} 
\usepackage{amssymb} 
\usepackage{exscale} 
\usepackage{amscd} 
\usepackage{latexsym} 
\usepackage{textcomp} 
\usepackage{bm}  
\usepackage{colortbl}
\usepackage{color}
\usepackage{fancyhdr}
\usepackage{cancel}
\usepackage{float}
\usepackage{lipsum}
\usepackage{tcolorbox}

\allowdisplaybreaks
\numberwithin{equation}{section}

\newtheorem{stat}{Statement}[section]
\newtheorem{theorem}[stat]{Theorem}

\newtheorem{proposition}[stat]{Proposition}
\newtheorem{lemma}[stat]{Lemma}
\newtheorem{definition}[stat]{Definition}
\newtheorem{remark}[stat]{Remark}

\newcommand{\IR}{{\mathbb R}}
\newcommand{\R}{{\mathbb R}}
\newcommand{\IN}{{\mathbb N}}

\newcommand{\beq}{\begin{equation}}
\newcommand{\eeq}{\end{equation}}
\newcommand{\bal}{\begin{align}}
\newcommand{\eal}{\end{align}}
\newcommand{\beqn}{\begin{equation*}}
\newcommand{\eeqn}{\end{equation*}}
\newcommand{\baln}{\begin{align*}}
\newcommand{\ealn}{\end{align*}}

\newcommand{\Pb}{\mathbb P}

\newcommand{\E}{\mathbb{E}}

\allowdisplaybreaks

\newcommand{\helv}{
\fontfamily{phv}\fontseries{b}\fontsize{9}{11}\selectfont}

 \DeclareGraphicsExtensions{.pdf,.png,.jpg}

\floatstyle{boxed}
\restylefloat{figure}

\title{A spectral-based numerical method for Kolmogorov equations in Hilbert spaces.}
\author{Francisco Delgado-Vences\footnote{Istituto Nazionale di Geof\'isica e Vulcanolog\'ia,
Pisa, Italy. Email addresses: francisco.delgado@ingv.it,\newline delgadovences@gmail.com}\hspace*{0.1cm} and Franco Flandoli\footnote{Dipartimento
di Matematica, Universit\'a di Pisa, Largo Bruno Pontecorvo 5,  56127 Pisa, Italy. Email address:
flandoli@dma.unipi.it.}.}

\begin{document}

\maketitle

\begin{abstract}

We propose a numerical solution for the solution of the Fokker-Planck-Kolmogorov (FPK) equations associated with stochastic
partial differential equations in Hilbert spaces. 
 The method is based on the spectral decomposition of the Ornstein-Uhlenbeck semigroup associated to the
Kolmogorov equation. This allows us to write the solution of the Kolmogorov equation as a deterministic version of the Wiener-Chaos
Expansion. By using this expansion we reformulate the Kolmogorov equation as a infinite system of ordinary differential equations, 
and by truncation it we set a linear finite system of differential equations. The solution of such system allow us to build an 
approximation to the solution of the Kolmogorov equations.  We test the numerical method with the Kolmogorov equations associated 
with a stochastic diffusion equation, a Fisher-KPP stochastic equation and a stochastic Burgers Eq. in dimension 1. 
\end{abstract}

\section{Introduction}

Stochastic Partial Differential Equations (SPDE’s) are important tools in modeling complex phenomena, they arise in many fields of
knowledge like Physics, Biology, Economy, Finance, etc.. Develop efficient numerical methods for simulating SPDE’s is very importan 
but also very difficult and challenging.\medskip

There exists in literature several approachs in order to solve numerically a SPDE. Among them there exists Monte Carlo simulations, 
Karhunen-Loeve Expansion, Wiener Chaos expansion, stochastic Taylor approximations for SPDE’s, etc. In order to solve 
numerically an SPDE one can apply one of this methods. Here we will mention some of them but our list reference is 
far away to be exhaustive.\medskip

The Monte Carlo (MC) simulations for SPDE's have been explored intensively in the last 20 years (\cite{pl}, \cite{kl-pl}  ). The basis
idea of MC is to sample the randomness in the SPDE's and solve the stochastic equations realization by realization, this is because 
for each given realization of the randomness, the SPDE's becomes deterministic and can be solved by usual deterministic numerical 
methods. the disadvatage is that many of that ``samples'' are required for suffcient accuracy, causing suboptimal efficiency even if 
optimal algebraic solvers are used; to overcome this issue Giles (\cite{gi}, \cite{gi1}) has introduce a modificacion of MC for the 
numerical solution of It\^o stochastic ordinary differential equations, following basic ideas in earlier work by S. Heinrich 
\cite{he} on numerical quadrature. This method is the so-called Multilevel Monte Carlo (MLMC).\medskip 

The main idea of MLMC methods is to apply the MC method for a nested sequence of stepsizes while balancing the number of samples 
according to the stepsize. MLMC allows to significantly speed up to classical MC methods thanks to this hierarchical sampling; 
however this method still can have limitations for SPDE's. \medskip

Other approach is use spectral methods, in particular use the Karhunen-Loeve expansion (KLE) and the Wiener Chaos expansions for
solving SPDE’s. For the former one can study the theory developed in \cite{le-kn}; in this work they proposed several methods to
solve SPDE’s and this methods are latter applied to solve the stochastic Navier-Stokes equations. Nevertheless, this method can
have limitations since in this approach the source of randomness is usually represented by a fixed number of random variables and
if we consider, for instance, stochastic equations arising in fluid dynamics with a random forcing white in time which has a 
divergent Karhunen-Loeve expansion, then it is not possible to apply the KLE to this kind of equations.\medskip

Hou et al \cite{ho-lu-ro-zh} propose a numerical method based on Wiener Chaos expansion and apply it to solve the stochastic
Burgers and Navier-Stokes equations driven by Brownian motion. They consider an SPDE with Brownian motion forcing and
since a Brownian motion can be expand as a linear combination of independent Gaussian random variables, then they expand the solution
of theSPDE's as a Fourier-Hermite series of those Gaussian random variables, this is a version of the Cameron-Martin decomposition.\medskip

There is another approach that involves stochastic Taylor approximations for stochastic partial differential equations (see
\cite{je-kl}). This Taylor expansions are based on an iterated application of the It\^o formula. However, For the solutions of 
stochastic partial differential equations in Hilbert (or Banach spaces) there is no way to define directly the It\^o formula.
Nevertheless, it can be constructed by taking advantage of the mild form representation of the solutions.\medskip

The  Fokker-Planck-Kolmogorov (FPK) equation is a partial differential 
equation that describes the time evolution of the probability density function of the velocity of a particle under the influence of 
drag forces and random forces, it is a kind of continuity equation for densities. Citing \cite{da-za} ``parabolic equations on Hilbert 
spaces appear in mathematical physics to model systems with infinitely many degrees of freedom. Typical examples are provided by spin
configurations in statistical mechanics and by crystals in solid state theory. Infinite-dimensional parabolic equations provide an 
analytic description of infinite dimensional diffusion processes in such branches of applied mathematics as population biology, fluid
dynamics, and mathematical finance.''. This kind of equations have been deeply studied in the last years, see for instance 
\cite{bo-da-ro}, \cite{da-fl-ro}, \cite{da} and the references therein.\\

Numerical methods for FPK equations associated with SPDEs have been studied, up to our knowledge, just in a few articles, here we will 
mention just one. Schwab and Suli \cite{sc-su} have formulated a space-time variational method to approximate solution of 
Kolmogorov-type equations in infinite dimensions. They consider an infinite-dimensional Hilbert space $\mathcal{H}$, a Gaussian 
measure $\mu$ with trace class covariance operator $Q$ on $\mathcal{H}$ and  the space $L^2(H,\mu)$ of functions on $\mathcal{H}$ 
which are square-integrable with respect to the measure $\mu$. They showed the well-posedness of Fokker-Plank equations and 
Ornstein-Uhlenbeck equations on $L^2(H,\mu)$. Moreover, they constructed sequences of finite-dimensional approximations that attain
the best possible convergence rates afforded by best $N$-terms approximations of the solution. They used an spectral method based on 
Wiener-Hermite polynomial chaos expansions in terms of a sequence of independent Gaussian random variables on $\mathcal{H}$ and a 
Wavelet type Riesz basis with respect to the time variable. The use of the spectral basis of Wiener-Hermite polynomial chaos allow 
them to avoid meshing the infinite-dimensional “domain” $\mathcal{H}$ of solutions of the Kolmogorov-type equations. However they do
not present numerical examples and the questions about the feasibility of their method are open.\medskip

In this paper, we introduce a novel numerical method that can have some similitude with the  one proposed by Schwab and Suli but also 
have substantial differences. Indeed, our method is also based on spectral methods for the variable on $\mathcal{H}$, but we use 
a deterministic version of the Wiener-Chaos Expansion on the infinite-dimensional ``domain'' $\mathcal{H}$ instead the classical 
Wiener-Chaos Expansion with the use of a sequence of Gaussian random variables, this allow allow us to avoid 
meshing the space  $\mathcal{H}$ but we also avoid the so-called curse of dimensionality: the associated computational cost grows 
exponentially as a function of the number of random variables defining the underlying probability space of the problem 
(see \cite{do-ia} for instance). \medskip

The second difference is with respect to the time variable, where, instead using  Wavelet type Riesz basis we set up a finite system
of coupled ordinary differential equations and by solving it we fix the coefficients as a time functions. we have applied the method 
to three SPDE's: a stochastic diffusion, a stochastic FisherKPP equation and a stochastic burgers equation and the results show that
the behaviour of the method is good. However, since the method is analogue to the classical deterministic spectral method thus it can 
be extended to improve its performance. This is the subject of a future research.\medskip

This paper is organized as follows. In section \ref{fpk-sect} we review the Fokker-Plank-Kolmogorov equation associated with SPDE's in
a separable Hilbert space. In section \ref{OUS-sect} we study the spectral decomposition of the Ornstein-Uhlenbeck semigroup
associated to the Kolmogorov equation which will be used to do the numerical approximation to the solution of 
the FPK equation, this is done in section \ref{NA-sect}. In section \ref{WPC-sect} we prove a theorem on the well posedness 
and  convergence of the numerical aproximation. Results on the application of the proposed method are presented in section 
\ref{AM-sect}, where we have applied the method to a linear stochastic diffusion equation, a Fisher-KPP stochastic equation and 
a stochastic Burgers Eq. in dimension 1.

\section{Fokker-Plank-Kolmogorov equation}\label{fpk-sect}
%{\color{red} A better introduction to FPK equations is missed}\\

In a separable infinite-dimensional Hilbert space $\mathcal{H}$ with inner product $\langle ,  \rangle_\mathcal{H}  $ we define a
Gaussian measure $\mu$ with mean zero and nuclear covariance operator $\Lambda$ with  $Tr(\Lambda)<+\infty$.

We focus on the stochastic differential equation in $\mathcal{H}$ 
\begin{equation}
\label{P1s2.1}
 dX_t=AX_tdt+B(X_t)dt+\sqrt{Q}dW_t,
\end{equation}
where the operator $A:\mathcal{D}(A)\subset \mathcal{H}\rightarrow \mathcal{H}$ is the infinitesimal generator of a strongly 
continuous semigroup $e^{tA}$ in $\mathcal{H}$, $Q$ is a bounded operator from another Hilbert space $\mathcal{U}$ to $\mathcal{H}$ 
and $B:\mathcal{D}(B)\subset \mathcal{H}\rightarrow \mathcal{H}$ is a nonlinear mapping. 

The equation \eqref{P1s2.1} can be associated to a Kolmogorov equation in the next way, we define 
\begin{equation}
\label{P1s2.2}
u(t,x)=\E\big[u_0(X_t^x)\big], 
\end{equation}
where $u_0:\mathcal{H}\rightarrow \IR$ and $X_t^x$ is the solution to \eqref{P1s2.1} with initial conditions $X_0=x$ where 
$x\in\mathcal{H}$. Then $u$ satisfies the Kolmogorov equation
\begin{equation}
\label{P1s2.3}
\frac{\partial u}{\partial t}= \frac{1}{2}Tr(QD^2u)+ \langle Ax, Du \rangle_\mathcal{H} + \langle B(x),Du \rangle_\mathcal{H},\qquad x\in D(A).
\end{equation}

Several authors have proved results on existence and uniqueness of the solution of the Kolmogorov equations, see for instance 
Da Prato \cite{da} for a survey,  Da Prato-Debussche \cite{da-de} for the Burgers equation,  Barbu-Da Prato \cite{ba-da} for the 
 2D Navier-Stokes stochastic flow in a channel.

\section{On the Ornstein-Uhlenbeck semigroup}\label{OUS-sect}

Following \cite{liu},  in $\mathcal{H}$ we define a Gaussian measure $\mu$ with mean zero and nuclear covariance operator $\Lambda$ with  
$Tr(\Lambda)<+\infty$ and since $\Lambda:\mathcal{H}\mapsto \mathcal{H}$ is a positive definite, self-adjoint operator then its 
square-root operator $\Lambda^{1/2}$ is a positive definite, self-adjoint Hilbert-Schmidt operator on $\mathcal{H}$.

Define the inner product 
\[
\langle g, h \rangle_0 := \langle \Lambda^{-1/2}g ,  \Lambda^{-1/2} h\rangle_\mathcal{H}, \quad \hbox{\rm for}\quad g,h\in \Lambda^{1/2} \mathcal{H}.
\]
Let $\mathcal{H}_0$ denote the Hilbert subspace of $\mathcal{H}$, which is the completion of $\Lambda^{1/2} \mathcal{H}$ with 
respect to the norm $\|g\|_0:= \langle g, g \rangle_0^{1/2} $. Then ${\mathcal{H}_0}$ is dense in $\mathcal{H}$ and the inclusion
map $i:\mathcal{H}_0\hookrightarrow\mathcal{H}$ is compact. The triple $(i,\mathcal{H}_0,\mathcal{H})$ forms an abstract Wiener space. 

Let $\mathbb{H}=L^2( \mathcal{H},\mu)$ denote the Hilbert space of Borel measurable functionals on the probability space with inner 
product 
\[
 \langle\Phi,\Psi\rangle_\mathbb{H}:=\int_{\mathcal{H}} \Phi(v) \Psi(v)\mu(dv),\quad\hbox{\rm for}\quad \Phi,\Psi\in\mathbb{H}, 
\]
and norm $\|\Phi\|_{\mathbb{H}}:=\langle\Phi,\Phi\rangle_\mathbb{H}^{1/2}$. In $\mathbb{H}$ we choose a basis system $\{\varphi_k\}$ 
such that $\varphi_k\in \mathcal{H}$.

A functional $\Phi:\mathcal{H}\mapsto \IR$, is said to be a smooth simple functional (or a cylinder functional) if there exists a 
$C^\infty$-function $\phi$ on $\IR^n$ and $n$-continuous linear functional $l_1,\ldots,l_n$ on $\mathcal{H}$ such that for
$h\in\mathcal{H}$
\[
 \Phi(h)=\phi(h_1,\ldots,h_n)\quad
\mbox{\rm where}\qquad h_i=l_i(h),\quad i=1,\ldots,n.
\]
The set of all such functionals will be denoted by $\mathcal{S}(\mathbb{H})$. \\

Denote by $P_k(x)$ the Hermite polynomial of degree $k$ taking values in $\IR$. Then, $P_k(x)$ is given by the following formula
\[
 P_k(x)=\frac{(-1)^k}{(k!)^{1/2}} e^{\tfrac{x^2}{2}} \frac{d^k}{dx^k}e^{-\tfrac{x^2}{2}}
\]
with $P_0=1$. It is well-known that $\{P_k(\cdot)\}_{k\in\IN}$ is a complete orthonormal system for $L^2(\IR,\mu_1(dx))$ with
$\mu_1(dx)=\tfrac{1}{\sqrt{2\pi}}e^{-\tfrac{x^2}{2}}dx$.

Define the set of infinite multi-index as
\[
 \mathcal{J}=\Big\{\bm{\alpha}=(\alpha_i,i\ge 1)\quad \big|\quad\alpha_i\in \IN\cup\{0\},\quad |\bm{\alpha}|:=\sum_{i=1}^\infty
 \alpha_i<+\infty  \Big\}
\]

For $\bm{n}\in\mathcal{J} $ define the {\it Hermite polynomial functionals} on $\mathcal{H}$ by 
\begin{align}
\label{s1.2}
H_{\bm{n}}(h) = \prod_{i=1}^\infty P_{n_i}(l_i(h)),\quad h\in\mathcal{H}_0,\quad \bm{n}\in \mathcal{J},
\end{align}
and where 
\[
l_i(h)=\langle h,  \Lambda^{-1/2}\varphi_i \rangle_\mathcal{H}, \quad i=1,2,\ldots
\]

where $P_n(\xi)$ is the usual Hermite polynomial for  $\xi\in\R$ and $n\in\IN$.

\begin{remark}
%By using the isomorphisms between infinite dimensional Hilbert spaces 
Notice that $l_i(h)$ is defined only for $h \in\mathcal{H}_0$. However, regarding $h$ as a $\mu$-random variable in
$\mathcal{H}$, we have $\E\big(l_i (h)\big) = \|\varphi_i \|^2  = 1$ and then $l_k (h)$ can be defined $\mu$-a.e. $h \in\mathcal{H}$,
similar to defining a stochastic integral.

It is possible to identify the Hermite polynomial functionals
defined in \eqref{s1.2}, for $h \in\mathcal{H}_0$, as a deterministic version of the Wick polynomials defined on the canonical Wiener space.(for further 
details see \cite{im} for instance).
\end{remark}

We have the following result (See Theorems 9.1.5 and 9.1.7 in Da Prato-Zabczyk \cite{da-za} or Lemma 3.1 in chapter 9 from 
Chow \cite{liu}).
\begin{lemma}\label{s1.le1}
 For $h\in\mathcal{H}  $ let $l_i(h)=\langle h,  \Lambda^{-1/2}\varphi_i \rangle_\mathcal{H}$, $ i=1,2,\ldots$. Then the set $\{H_{\bm{n}}\}$ of all Hermite polynomials on $\mathcal{H} $ forms a complete orthonormal system for $\mathbb{H} $. Hence the set of all functionals are dense in $\mathbb{H} $. Moreover, we have the direct sum decomposition:
 \[
\mathbb{H} = \bigoplus_{j=0}^\infty K_j,
\]
where $K_j$ is the subspace of $\mathbb{H} $ spanned by $\{H_{\bm{n}}: |\bm{n}|=j\}$.\hfill$\Box$ 
\end{lemma}

\subsection*{Spectral decomposition of the Ornstein-Uhlenbeck semigroup}

Consider the linear stochastic equation 
\begin{align}
 du_t&=Au_tdt+dW_t,\label{OU}\\
 u_0&=h\in \mathcal{H}.\nonumber
\end{align}
Here, as before  $A:\mathcal{D}(A)\subset \mathcal{H}\rightarrow \mathcal{H}$ is the infinitesimal generator of a strongly 
continuous semigroup $e^{tA}$ in $\mathcal{H}$. 
$W_t$ is a $Q$-Wiener process in $\mathcal{H}$.

Chow in \cite[Lemma 9.4.1]{liu} has shown the following result.
\begin{lemma}
Suppose that $A$ and $Q$ satisfy the following: 
\begin{enumerate}
 \item $A:\mathcal{D}(A)\subset \mathcal{H}\rightarrow \mathcal{H}$ is self-adjoint and there is $\beta>0$ such that
 \[
  \langle Av,v\rangle_ \mathcal{H}\le -\beta\|v \|_\mathcal{H}\quad \forall v\in \mathcal{H}.
 \]
\item $A$ commutes with $Q$ in $\mathcal{D}(A)\subset \mathcal{H}$.
\end{enumerate}
Then \eqref{OU} has a unique invariant measure $\mu$ which is a Gaussian measure on $ \mathcal{H}$ with zero mean and covariance 
operator $\Lambda=\tfrac{1}{2}Q(-A)^{-1}=\tfrac{1}{2}(-A)^{-1}Q$.\hfill$\Box$
\end{lemma}

 We define the operator 
\begin{equation}
\label{P1s2.4}
\mathcal{A}_0 u= \frac{1}{2}Tr(QD^2u)+ \langle Ax, Du \rangle_\mathcal{H},\qquad\quad x\in\mathcal{H}
\end{equation}
and suppose that $-A$ and $Q$ have the same eigenfunctions $e_k$ with eigenvalues $\lambda_k$ y $\rho_k$ respectively. 
% With some aditional 
% hypothesis it is possible to obtain that $\Lambda=\frac{1}{2}Q(-A)^{-\tfrac{1}{2}}=\frac{1}{2}(-A)^{-\tfrac{1}{2}}Q$.

Then the operator $\mathcal{A}_0 $ satisfies the following result.
\begin{lemma}\label{s1.le2}
 Let $H_{\bm{n}}(h) $ be a Hermite polynomial functional given by \eqref{s1.2}. Then the following holds
\begin{equation}
\label{s2.5}
\mathcal{A}_0 H_{\bm{n}}(h)= -\lambda_{\bm{n}} H_{\bm{n}}(h),
\end{equation}
for any $\bm{n}\in \mathcal{J} $ and $h\in \mathcal{H} $ and where
\begin{equation*}
\lambda_{\bm{n}}:=\sum_{k=1}^\infty n_k\lambda_k.
\end{equation*}
\hfill$\Box$
\end{lemma}
The proof can be found in \cite[Lemma 9.4.3]{liu} or \cite{liu2}.

Using lemmas \ref{s1.le2} and \ref{s1.le1}, ($\{H_{\bm{n}}\}$ forms a complete orthonormal system for 
$L^2(\mathcal{H},\mu)$) we can write
\begin{equation}
\label{s3.1}
u(t,x) =  \sum_{\bm{n}\in \mathcal{J}} u_{\bm{n}}(t) H_{\bm{n}}(x), \qquad x\in\mathcal{H},\quad t\in[0,T],
\end{equation}
where $u_{\bm{n}}:[0,T]\mapsto \IR$ and $H_{\bm{n}}(x)$ are the Hermite functionals.

\begin{remark}
 The decomposition given in \eqref{s3.1} is a deterministic version to the Wiener Chaos expansion (WCe), also known as a 
 Fourier-Hermite series. The WCe has been used to prove several results in stochastic analysis and also it has been 
 applied to solve numerically stochastic partial differential equations (see for instance Lototsky and Rozovskii \cite{lo-ro}, 
 Lototsky \cite{lo},  Hou et. al. \cite{ho-lu-ro-zh}).
\end{remark}

Notice that the Kolmogorov equation can be writen as
\begin{align}
\frac{\partial u}{\partial t}&= \frac{1}{2}Tr(QD^2u)+ \langle Ax, Du \rangle_\mathcal{H} + \langle B(x),Du \rangle_\mathcal{H}\nonumber\\
&= \mathcal{A}_0u  + \langle B(x),Du \rangle_\mathcal{H}.\label{s3.2}
\end{align}
Using \eqref{s3.1}, we calculate 
\begin{align*}
\frac{\partial u}{\partial t}&=  \sum_{\bm{n}\in \mathcal{J}} \dot{u}_{\bm{n}}(t) H_{\bm{n}}(x)\\
\mathcal{A}_0u  &= \mathcal{A}_0 \sum_{\bm{n}\in \mathcal{J}} u_{\bm{n}}(t) H_{\bm{n}}(x) =  \sum_{\bm{n}\in \mathcal{J}} u_{\bm{n}}(t) \mathcal{A}_0H_{\bm{n}}(x)\\
&=  -\sum_{\bm{n}\in \mathcal{J}} u_{\bm{n}}(t) \lambda_{\bm{n}} H_{\bm{n}}(x)
\end{align*}
Where in the last equality we have used the Lemma \ref{s1.le2}.

For the last term in \eqref{s3.2} we have
\begin{align*}
 \langle B(x),Du \rangle_\mathcal{H} &=  \Big\langle B(x),D_x \sum_{\bm{n}\in \mathcal{J}} u_{\bm{n}}(t) H_{\bm{n}}(x) \Big\rangle_\mathcal{H}\\
 &= \sum_{\bm{n} \in \mathcal{J}} u_{\bm{n}}(t) \big\langle B(x),D_x  H_{\bm{n}}(x) \big\rangle_\mathcal{H},
\end{align*}
where $D_x$ is the Fr\'echet derivative.

Therefore the Kolmogorov equation becomes
\begin{align*}
\sum_{\bm{n}\in \mathcal{J}} \dot{u}_{\bm{n}}(t) H_{\bm{n}}(x)&= -\sum_{\bm{n}\in \mathcal{J}} u_{\bm{n}}(t) \lambda_{\bm{n}} H_{\bm{n}}(x) 
+ \sum_{\bm{n} \in \mathcal{J}} u_{\bm{n}}(t) \big\langle B(x),D_x  H_{\bm{n}}(x) \big\rangle_\mathcal{H}
\end{align*}

Multiplying by $H_{\bm{m}}(x) $, $\bm{m}\in\mathcal{J} $ and integrating in $ \mathcal{H}$ w.r.t $\mu(dx)$ we have
\begin{align*}
\sum_{\bm{n}\in \mathcal{J}} \dot{u}_{\bm{n}}(t)\int_{\mathcal{H}} H_{\bm{m}}(x)  H_{\bm{n}}(x)\mu(dx) &= -\sum_{\bm{n}\in \mathcal{J}} u_{\bm{n}}(t) 
\lambda_{\bm{n}} \int_{\mathcal{H}} H_{\bm{m}}(x) H_{\bm{n}}(x) \mu(dx)\\
& + \sum_{\bm{n}\in \mathcal{J}} u_{\bm{n}}(t)  \int_{\mathcal{H}} H_{\bm{m}}(x) \big\langle B(x),D_x  H_{\bm{n}}(x) \big\rangle_\mathcal{H} \mu(dx)
\end{align*}

 From this, and using the orthogonality of the system $\{H_{\bm{m}}(x) \}$ we get the infinite system of coupled ordinary differential equations
\begin{equation}
   \dot{u}_{\bm{m}}(t) =   -u_{\bm{m}}(t) \lambda_{\bm{m}} 
  + \sum_{\bm{n}\in \mathcal{J}} u_{\bm{n}}(t) C_{\bm{n},\bm{m}},\qquad \bm{n},\bm{m}\in\mathcal{J}\label{inf-sys}
\end{equation}
where $ C_{\bm{n},\bm{m}}$ is given by
\begin{equation}
C_{\bm{n},\bm{m}}:=\int_\mathcal{H} \big\langle B(x), D_x H_{\bm{n}}(x) \big\rangle_\mathcal{H} H_{\bm{m}}(x) \mu(dx).\label{C-NM}
 \end{equation}
We focus on $C_{\bm{n},\bm{m}}$. Since $ H_{\bm{n}}(x)=  \prod_{i=1}^\infty P_{n_i}\big(\langle x,\Lambda^{-\tfrac{1}{2}}
e_i\rangle_\mathcal{H}  \big)$ we get
\begin{align*}
D_x H_{\bm{n}}(x)&=  \sum_{k=1}^\infty\prod_{\stackrel{i=1}{i\ne k}}^\infty P_{n_i}\big(\langle x,\Lambda^{-\tfrac{1}{2}} 
e_i\rangle_\mathcal{H}\big) P_{n_k}^{'} \big(\langle x,\Lambda^{-\tfrac{1}{2}} e_k\rangle_\mathcal{H}\big) 
\Lambda^{-\tfrac{1}{2}} e_k,
% & = \sum_{k=1}^\infty \frac{H_{\bm{n}}(x)}{ P_{n_k}\big(\langle x,\Lambda^{-\tfrac{1}{2}} e_k\rangle_\mathcal{H}\big) }  
% P_{n_k}^{'}\big(\langle x,\Lambda^{-\tfrac{1}{2}} e_k\rangle_\mathcal{H}\big) \Lambda^{-\tfrac{1}{2}} e_k
\end{align*}
then
\begin{align*}
 \big\langle B(x),  D_x H_{\bm{n}}(x)  \big\rangle_\mathcal{H}&= \sum_{k=1}^\infty \Big\langle B(x),
 \Lambda^{-\tfrac{1}{2}} e_k\Big\rangle_\mathcal{H}  \prod_{\stackrel{i=1}{i\ne k}}^\infty P_{n_i}\big(\langle x,
 \Lambda^{-\tfrac{1}{2}} e_i\rangle_\mathcal{H}\big) P_{n_k}^{'} \big(\langle x,\Lambda^{-\tfrac{1}{2}} e_k\rangle_\mathcal{H}\big)
% &= \sum_{k=1}^\infty \Big\langle B(x),  \Lambda^{-\tfrac{1}{2}} e_k\Big\rangle_\mathcal{H} 
%  \frac{ H_{\bm{n}}(x)}{ P_{n_k}\big(\langle x,\Lambda^{-\tfrac{1}{2}} e_k\rangle_\mathcal{H}\big) }
%  P_{n_k}^{'} \big(\langle x,\Lambda^{-\tfrac{1}{2}} e_k\rangle_\mathcal{H}\big) \\
% &= H_{\bm{n}}(x) \sum_{k=1}^\infty \Big\langle B(x),  \Lambda^{-\tfrac{1}{2}} e_k\Big\rangle_\mathcal{H} 
%  \frac{P_{n_k}^{'} \big(\langle x,\Lambda^{-\tfrac{1}{2}} e_k\rangle_\mathcal{H}\big) }
%  { P_{n_k}\big(\langle x,\Lambda^{-\tfrac{1}{2}} e_k\rangle_\mathcal{H}\big) }
  \end{align*}
Thus, 
\begin{align*}
 C_{\bm{n},\bm{m}}=\int_\mathcal{H} \sum_{k=1}^\infty \Big\langle B(x),  \Lambda^{-\tfrac{1}{2}} e_k\Big\rangle_\mathcal{H} 
 \prod_{\stackrel{i=1}{i\ne k}}^\infty P_{n_i}\big(\langle x,  \Lambda^{-\tfrac{1}{2}} e_i\rangle_\mathcal{H}\big) P_{n_k}^{'} 
 \big(\langle x,\Lambda^{-\tfrac{1}{2}} e_k\rangle_\mathcal{H}\big) H_{\bm{m}}(x) \mu(dx).
 \end{align*}
 
 \subsubsection*{A technical result}
 The following result is important for the numerical simulation since it will allow us to use the evaluation functional on the 
 Hilbert space $\mathcal{H}$.
 
\begin{lemma}
i) The Gaussian measure $\mu$ on $\mathcal{H}=L^{2}\left(  0,1\right)  $ with
covariance $\Lambda=\frac{1}{2}\left(  -A\right)  ^{-1}$ is supported on
$C\left(  \left[  0,1\right]  \right)  $.

ii)\ Let $\xi_{0}\in\left[  0,1\right]  $ be given. Let $u_{0}:C\left(
\left[  0,1\right]  \right)  \rightarrow\mathbb{R}$ be defined as
$u_{0}\left(  x\right)  =x\left(  \xi_{0}\right)  $. Then%
\[
\int_{\mathcal{H}}u_{0}^{2}\left(  x\right)  \mu\left(  dx\right)  <\infty
\]
(and therefore $\sum_{m}\left(  u_{m}^{0}\right)  ^{2}<\infty$).
\end{lemma}

\begin{proof}
Recall that $Af=f^{\prime\prime}$, $D\left(  A\right)  =H^{2}\left(
0,1\right)  \cap H_{0}^{1}\left(  0,1\right)  $. By solving the two-point
boundary value problem $f^{\prime\prime}=g$, $f\left(  0\right)  =f\left(
1\right)  =0$, after several manipulations one can show that
\[
\left(  \Lambda h\right)  \left(  \xi\right)  =\int_{0}^{1}\lambda\left(
\xi,\xi^{\prime}\right)  h\left(  \xi^{\prime}\right)  d\xi^{\prime},\qquad
h\in\mathcal{H}%
\]
where
\[
\lambda\left(  \xi,\xi^{\prime}\right)  =\frac{1}{2}\left[  \left(  \xi\left(
1-\xi^{\prime}\right)  \right)  -\left(  \xi-\xi^{\prime}\right)
1_{\xi^{\prime}\leq\xi}\right]  .
\]
The reader may more easily get convinced that this is correct a posteriori, by
showing that $\frac{d^{2}}{d\xi^{2}}\int_{0}^{1}\lambda\left(  \xi,\xi
^{\prime}\right)  f\left(  \xi^{\prime}\right)  d\xi^{\prime}=-\frac{1}%
{2}f\left(  \xi\right)  $ and that $\left(  \Lambda h\right)  \left(
0\right)  =\left(  \Lambda h\right)  \left(  1\right)  =0$. 

Consider the canonical process $\left(  X_{\xi}\right)  _{\xi\in\left[
0,1\right]  }$:$\left(  \mathcal{H},\mathcal{B}\left(  \mathcal{H}\right)
,\mu\right)  \rightarrow\left(  \mathbb{R},\mathcal{B}\left(  \mathbb{R}%
\right)  \right)  $ defined for a.e. $\xi\in\left[  0,1\right]  $ as $X_{\xi
}\left(  x\right)  =x\left(  \xi\right)  $ and denote by $E$ the mathematical
expectation on $\left(  \mathcal{H},\mathcal{B}\left(  \mathcal{H}\right)
,\mu\right)  $. The process $X$ has zero mean. One can prove that%
\[
Cov\left(  X_{\xi},X_{\xi^{\prime}}\right)  =q\left(  \xi,\xi^{\prime}\right)
,\qquad\text{a.e. }\xi,\xi^{\prime}\in\left[  0,1\right]  .
\]
Indeed, since $\left\langle \Lambda h,k\right\rangle _{\mathcal{H}}%
=\int_{\mathcal{H}}\left\langle x,h\right\rangle _{\mathcal{H}}\left\langle
x,k\right\rangle _{\mathcal{H}}\mu\left(  dx\right)  $, we have
\begin{align*}
\int_{0}^{1}\int_{0}^{1}q\left(  \xi,\xi^{\prime}\right)  h\left(  \xi
^{\prime}\right)  k\left(  \xi\right)  d\xi d\xi^{\prime}  & =\int%
_{\mathcal{H}}\int_{0}^{1}x\left(  \xi^{\prime}\right)  h\left(  \xi^{\prime
}\right)  d\xi^{\prime}\int_{0}^{1}x\left(  \xi\right)  k\left(  \xi\right)
d\xi\mu\left(  dx\right)  \\
& =\int_{0}^{1}\int_{0}^{1}\left(  \int_{\mathcal{H}}x\left(  \xi^{\prime
}\right)  x\left(  \xi\right)  \mu\left(  dx\right)  \right)  h\left(
\xi^{\prime}\right)  k\left(  \xi\right)  d\xi d\xi^{\prime}\\
& =\int_{0}^{1}\int_{0}^{1}E\left[  X_{\xi^{\prime}}X_{\xi}\right]  h\left(
\xi^{\prime}\right)  k\left(  \xi\right)  d\xi d\xi^{\prime}%
\end{align*}
and the formula for $Cov\left(  X_{\xi},X_{\xi^{\prime}}\right)  $ follows
from the arbitrarity of $h$ and $k$. 

The paths of the process $X$ are obviously of class $L^{2}\left(  0,1\right)
$; there is a continuous modification of $X$ if and only if $\mu$ is supported on
$C\left(  \left[  0,1\right]  \right)  $. If we check the condition%
\begin{equation}
E\left[  \left\vert X_{\xi}-X_{\xi^{\prime}}\right\vert ^{2}\right]  \leq
C\left\vert \xi-\xi^{\prime}\right\vert ^{\alpha}\label{Kolm}%
\end{equation}
for some $\alpha,C>0$, then, by gaussianity,
\[
E\left[  \left\vert X_{\xi}-X_{\xi^{\prime}}\right\vert ^{p}\right]  \leq
C_{p}\left\vert \xi-\xi^{\prime}\right\vert ^{\alpha p/2}%
\]
for every $p\geq1$ and for a suitable constant $C_{p}>0$, hence there is a
continuous modification by Kolmogorov criterion. But
\begin{align*}
E\left[  \left\vert X_{\xi}-X_{\xi^{\prime}}\right\vert ^{2}\right]    &
=E\left[  X_{\xi}^{2}\right]  +E\left[  X_{\xi^{\prime}}^{2}\right]
-2E\left[  X_{\xi}X_{\xi^{\prime}}\right]  \\
& =q\left(  \xi,\xi\right)  +q\left(  \xi^{\prime},\xi^{\prime}\right)
-2q\left(  \xi,\xi^{\prime}\right)  \\
& =q\left(  \xi,\xi\right)  -q\left(  \xi,\xi^{\prime}\right)  \\
& +q\left(  \xi^{\prime},\xi^{\prime}\right)  -q\left(  \xi,\xi^{\prime
}\right)  .
\end{align*}
We have%
\begin{align*}
& q\left(  \xi,\xi\right)  -q\left(  \xi,\xi^{\prime}\right)  \\
& =\frac{1}{2}\xi\left(  1-\xi\right)  -\frac{1}{2}\left[  \left(  \xi\left(
1-\xi^{\prime}\right)  \right)  -\left(  \xi-\xi^{\prime}\right)
1_{\xi^{\prime}\leq\xi}\right]  \\
& \leq C\left\vert \xi-\xi^{\prime}\right\vert
\end{align*}
and similarly $q\left(  \xi^{\prime},\xi^{\prime}\right)  -q\left(  \xi
,\xi^{\prime}\right)  \leq C\left\vert \xi-\xi^{\prime}\right\vert $. Hence
condition (\ref{Kolm}) is satisfied with $\alpha=1$. We have proved that $\mu$
is supported on $C\left(  \left[  0,1\right]  \right)  $.

We also have%
\[
\int_{\mathcal{H}}u_{0}^{2}\left(  x\right)  \mu\left(  dx\right)
=\int_{\mathcal{H}}x^{2}\left(  \xi_{0}\right)  \mu\left(  dx\right)
=E\left[  X_{\xi_{0}}^{2}\right]  =q\left(  \xi_{0},\xi_{0}\right)  <\infty.
\]
The proof is complete. 
\end{proof}

\begin{remark}
To convince ourselves, a more concise but a little formal proof of the claim
$\int_{\mathcal{H}}u_{0}^{2}\left(  x\right)  \mu\left(  dx\right)  <\infty$
is%
\[
\int_{\mathcal{H}}u_{0}^{2}\left(  x\right)  \mu\left(  dx\right)
=\int_{\mathcal{H}}\left\langle x,\delta_{\xi_{0}}\right\rangle _{\mathcal{H}%
}^{2}\mu\left(  dx\right)  =\left\langle \Lambda\delta_{\xi_{0}},\delta
_{\xi_{0}}\right\rangle _{\mathcal{H}}=\frac{1}{2}\left\Vert \left(
-A\right)  ^{-1/2}\delta_{\xi_{0}}\right\Vert _{\mathcal{H}}^{2}<\infty
\]
because $\left(  -A\right)  ^{-1/2}\delta_{\xi_{0}}\in L^{2}\left(
0,1\right)  $, since by duality
\begin{align*}
\left\langle \left(  -A\right)  ^{-1/2}\delta_{\xi_{0}},f\right\rangle
_{\mathcal{H}}  & =\left\langle \delta_{\xi_{0}},\left(  -A\right)
^{-1/2}f\right\rangle _{\mathcal{H}}=\left(  \left(  -A\right)  ^{-1/2}%
f\right)  \left(  \xi_{0}\right)  \\
& \leq\left\Vert \left(  -A\right)  ^{-1/2}f\right\Vert _{L^{\infty}}\leq
C\left\Vert \left(  -A\right)  ^{-1/2}f\right\Vert _{H^{1}}\leq C\left\Vert
f\right\Vert _{L^{2}}%
\end{align*}
where we have used Sobolev embedding $H^{1}\subset L^{\infty}$ and the fact
that $\left(  -A\right)  ^{-1/2}$ maps $L^{2}$ into $H^{1}$.
\end{remark}

\section{Numerical approximation}\label{NA-sect}

Define the set of {\it finite} multi-index $J^{M,N}$ as

 \[
 \mathcal{J}^{M,N}=\Big\{\bm{\alpha}=(\alpha_i,1\le i\le M)\quad \big|\quad\alpha_i\in\{0,1,2,\ldots,N\} \Big\}
\]
this is the set of $M$-tuple wich can take values in the set $\{0,1,2,\ldots,N\} $. 
%For simplicity we wil take $M=N$ and we will denote by
%$\mathcal{J}^N=\mathcal{J}^{N,N}$.

We approximate the solutions of the Kolmogorov equation by the following expression
\begin{equation}
\label{s3.1.1}
 \hat{u}_N(t,x) =  \sum_{ \bm{n}\in \mathcal{J}^{M,N}} u_{\bm{n}}(t) H_{\bm{n}}(x), \qquad x\in\mathcal{H},\quad t\in[0,T],
\end{equation}
Notice the use of the finite $M$-tuple in oposition to the infinite multi-index $\mathcal{J}$ as in \eqref{s3.1}.
 
We truncate the infinite system \eqref{inf-sys} in the following sense. Consider the same value $M$ as in $J^{M,N}$  and $\bm{m}_1,\ldots,\bm{m}_M\in\mathcal{J}^{M,N}$ and
define the finite system of equations 
\begin{align}
  \dot{u}_{\bm{m}_i}(t) =   -u_{\bm{m}_i}(t) \lambda_{\bm{m}_i} 
  + \sum_{j=1}^M  u_{\bm{n}_j}(t) C_{\bm{n}_j,\bm{m}_i}.\qquad 1\le i\le M.\label{fin-sys}
\end{align}
Set the vectors
\begin{align*}
 U^M(t)&=\big( u_{\bm{m}_1}(t),u_{\bm{m}_2}(t),\ldots,u_{\bm{m}_M}(t) \big)^T\\
 \dot{U}^M(t)&=\big(\dot{u}_{\bm{m}_1}(t),\dot{u}_{\bm{m}_2}(t),\ldots,\dot{u}_{\bm{m}_M}(t) \big)^T
 \end{align*}
and the matrix 

\[ 
A=\left( \begin{array}{ccccc}
-\lambda_{1}+ C_{1,1} & C_{2,1} & \cdots & C_{M-1,1} & C_{M,1} \\
 C_{1,2} & -\lambda_{2}+C_{2,2} & \cdots & C_{M-1,2} & C_{M,2} \\
\vdots &\vdots&\ddots&\vdots&\vdots \\
C_{1,M-1} & C_{2,M-1} & \cdots & -\lambda_{M-1}+ C_{M-1,M-1} & C_{M,M-1} \\
C_{1,M} & C_{2,M} & \cdots & C_{M-1,M} & -\lambda_{M}+ C_{M,M} \\
\end{array} 
\right)
\] 
where $\lambda_{i}=\lambda_{\bm{m}_i}$ and $C_{i,j}=C_{\bm{n}_i,\bm{m}_j}$ for $1\le i,j\le M$. Notice that, given the expression 
\eqref{C-NM}, in general the matrix $A$ is not symmetric. We now can write the system \eqref{fin-sys} as a matrix differential
equation:
\begin{align}
 \dot{U}^M(t)=A U^M(t).\label{MDE}
\end{align}

Then, if $A$ has $M$ real and distint eigenvalues $\eta_i$ and $M$ eigenvectors $\vec{V}_i$ then the solution to the \eqref{MDE} is
given by
\begin{align}
U^M(t)=\sum_{i=1}^M c_i \vec{V}_ie^{\eta_it}.
\end{align}
 In the case when some of the eigenvalues and eigenvectors, or at least one of them, take values in the complex field we still can have real
 solutions. Indeed, Suppose that we have the case with one complex eigenvalue and eigenvector then it is know that we will have 
 $M-2$ real eigenvalues but we can obtain two real solutions from the complex eigenvalue(see \cite{go-sc} for instance).
 
 Let us write one of the complex eigenvalue and eigenvector as
 \begin{align*}
 \vec{V}&=\vec{a}+i\vec{b},\\ 
 \eta &=\gamma+i\mu,
 \end{align*}
then we can write two {\it real} solutions as follows:
\begin{align*}
 e^{\gamma t}\big(\vec{a}cos(\mu t)-\vec{b} sin(\mu t)\big),\qquad
  e^{\gamma t}\big(\vec{a}sin(\mu t)+\vec{b} cos(\mu t)\big).
\end{align*}

 \subsection{Initial Conditions}\label{ICsect} 

In contrast to several types of differential equations, whether ordinary or partial, deterministic or stochastic,
for FPK equations there is no standard way to determine the initial conditions. 
This is because in this type of equations we must choose a functional that acts on the initial condition, 
this implies that  depending on the functional chosen we must adapt the method. 
Here we present the method for two examples of functionals.

We will consider two cases :  
 \begin{align*}
  u_0^{z_0}(g)&:= g(z_0).\qquad \mbox{for fixed } z_0\in[0,1]\\
 &\mbox{\hspace*{-3cm} and }\\
 \vspace*{-0.1cm}  u_0(g)&:=\int_0^1 g(z) dz.\\
 \end{align*}

For the first functional, define the set points in the set $[a,b]$ as $\{z_i\}$, $i=0,\ldots,P$, such that 
$z_0=a$ and $z_{P}=b$. Then for each point $z_i$ we have that $X_0(z_i)=X(0,z_i)$, and for each $z_i$ set
$u_0(x)$ as the evaluation functional  $z_i\mapsto X_t^x(z_i)$ then from $u(t,x)=\E(u_0(X_t^x))$ we obtain
\[
u(0,x)=\E\big(u_0^{z_i}(X_0^x)\big)=X^x(0,z_i)=x(z_i),
\]
and at other hand
\[
u(0,x)=\sum_{\bm{n}\in \mathcal{J}^{M,N}} u_{\bm{n}}(0) H_{\bm{n}}(x),
\]
then for each $z_i$
\[
x(z_i)=u(0,x) =\sum_{\bm{n}\in \mathcal{J}^{M,N}} u_{\bm{n}}(0) H_{\bm{n}}(x)
\]
Then, multiplying by $ H_{\bm{m}}(x) $ and integrating in the Hilbert space $L^2(\mathcal{H},\mu)$ we have
\[
u_{\bm{m}}(0)=\int_{\mathcal{H}} x(z_i) H_{\bm{m}}(x) \mu(dx).
\]
Here the value of the initial condition $u_{\bm{m}}(0)$ depends on $z_i$, i.e. $u_{\bm{m}}(0)=u_{\bm{m}}^{z_i}(0) $. 

Notice that in the direction of the eigenfunction $e_k$ the expression $x$ can be writen as
$ \langle x, e_k\rangle_{\mathcal{H}} e_k$  and then we can write $H_{\bm{m}}(x) x(z_i) $ in the direction $e_k$ as
$   P_{m_k} \big(\xi_k\big) \langle x, e_k\rangle_{\mathcal{H}} e_k(z_i) $ with $\xi_k= \langle x, \Lambda^{-1/2} e_k\rangle_{\mathcal{H}}$. 
Furthermore, $\xi_k= \langle x, \Lambda^{-1/2} e_k\rangle_{\mathcal{H}}= |\lambda_k| \langle x, e_k\rangle_{\mathcal{H}}$ then we have

\begin{align}
 u_{\bm{m}}^{z_i}(0)&=\int_{\mathcal{H}} x(z_i) H_{\bm{m}}(x) \mu(dx)\nonumber\\
&= \int_{\IR^\IN} \sum_{k=1}^\infty e_k(z_i)\langle x, e_k\rangle_{\mathcal{H}} P_{m_k} \big(\xi_k\big)  
\mu(d\xi_1,d\xi_2,\cdots)e_k\nonumber\\
&= \int_{\IR^\IN} \sum_{k=1}^\infty e_k(z_i) \frac{\xi_k}{\lambda_k}P_{m_k} \big(\xi_k\big) \mu(d\xi_1,d\xi_2,\cdots)e_k\nonumber\\
&= \sum_{k=1}^\infty \frac{e_k(z_i)}{\lambda_k} \int_{\IR}  P_{m_k} \big(\xi_k\big) \xi_k \mu(d\xi_k)\nonumber\\
&\approx  \sum_{k=1}^M \frac{e_k(z_i)}{\lambda_k} \int_{\IR}  P_{m_k} \big(\xi_k\big) \xi_k \mu(d\xi_k)\label{ME-BC3}
\end{align}

Notice that the general solution to each $u_{\bm{m}}^{z_i}(0)$ is given by the expression

 \begin{equation*}
 \left( \begin{array}{c}
 u_1(t)\\
u_2(t) \\
\vdots \\
u_{M-1}(t) \\
u_M(t)\\
\end{array} 
\right)=\left(\mathbf{V}_1 \quad \mathbf{V}_2 \quad\cdots \quad\mathbf{V}_{M-1} \quad \mathbf{V}_M \right)
 \left( \begin{array}{c}
 c_1e^{\lambda_1 t }\\
 c_2e^{\lambda_2 t }\\
\vdots \\
 c_{M-1}e^{\lambda_{M-1} t }\\
 c_Me^{\lambda_M t }\\
\end{array} 
\right)
\end{equation*}
 where $\mathbf{V}_j$ and $\lambda_j$ are the eigenvector and eigenvalue of the matrix $A$ and we are denoting 
 $u_j(t)=u_{\bm{m}_j}^{z_i}(t)$,  $1\le j\le M$. Evaluating in $t=0$ we have 

 \begin{equation*}
 \left( \begin{array}{c}
 u_1(0)\\
u_2(0) \\
\vdots \\
u_{M-1}(0) \\
u_M(0)\\
\end{array} 
\right)=\Big(\mathbf{V}_1 \quad \mathbf{V}_2 \quad\cdots \quad\mathbf{V}_{M-1} \quad \mathbf{V}_M \Big)
 \left( \begin{array}{c}
 c_1\\
 c_2\\
\vdots \\
 c_{M-1}\\
 c_M\\
\end{array} 
\right),
\end{equation*} 
and therefore 
 \begin{equation*}
 \left( \begin{array}{c}
 c_1\\
 c_2\\
\vdots \\
 c_{M-1}\\
 c_M\\
\end{array} 
\right)=\Big(\mathbf{V}_1 \quad \mathbf{V}_2 \quad\cdots \quad\mathbf{V}_{M-1} \quad \mathbf{V}_M \Big)^{-1}
  \left( \begin{array}{c}
 u_1(0)\\
u_2(0) \\
\vdots \\
u_{M-1}(0) \\
u_M(0)\\
\end{array} 
\right),
\end{equation*} 
with $u_j(t)=u_{\bm{m}_j}^{z_i}(t)$ given by the expression \eqref{ME-BC3}. Now we are able to fix the value of the 
initial conditions for the first case. Notice that also the contants $c_j$ depend on the value $z_i$, i.e. 
$c_j=c_j^{z_i}$.\medskip

For the second functional, from $u(t,x)=\E(u_0(X_t^x))$ we obtain
\[
u(0,x)=\E(u_0(X_0^x))=\int_0^1 x(z) dz,
\]
and at other hand
\[
u(0,x)=\sum_{\bm{n}\in \mathcal{J}^{M,N}} u_{\bm{n}}(0) H_{\bm{n}}(x),
\]
then 
\[
\int_0^1 x(z) dz=\sum_{\bm{n}\in \mathcal{J}^{M,N}} u_{\bm{n}}(0) H_{\bm{n}}(x).
\]
Multiplying by $ H_{\bm{m}}(x) $ and integrating in the Hilbert space $L^2(\mathcal{H},\mu)$ and by using Fubini we have
\begin{align*}
 u_{\bm{m}}(0)=\int_{\mathcal{H}} \int_0^1 x(z) dz H_{\bm{m}}(x) \mu(dx)= \int_0^1\left( \int_{\mathcal{H}}x(z) H_{\bm{m}}(x) \mu(dx)\right)dz 
\end{align*}
We focus on the integral on $\mathcal{H}$. By following the steps given for the first functional (just replacing $z_i$ by $z$) we can arrive to the following expression
\[
  \int_{\mathcal{H}}x(z) H_{\bm{m}}(x) \mu(dx)
  \approx \prod_{k=1}^M \frac{e_k(z)}{\lambda_k} \int_{\IR}  P_{m_k} \big(\xi_k\big) \xi_k \mu(d\xi_k),
\]
thus 
\begin{align}
 u_{\bm{m}}(0) &\approx \int_0^1  \prod_{k=1}^M \frac{e_k(z)}{\lambda_k} 
 \left( \int_{\IR}  P_{m_k} \big(\xi_k\big) \xi_k \mu(d\xi_k)\right) dz\nonumber\\
&= \prod_{k=1}^M \int_{\IR}  P_{m_k} \big(\xi_k\big) \xi_k \mu(d\xi_k) \int_0^1   \frac{e_k(z)}{\lambda_k}  dz \label{ME-BC4}  
\end{align}

From here and by following the procedure for the first functional we are able to fix the initial conditions.

\section{Well posedness and convergence}\label{WPC-sect}

Let $\mathcal{J}$ be a countable set, $\left\{  \lambda_{m};m\in
\mathcal{J}\right\}  $ a sequence of positive real numbers diverging to
infinity and $\left\{  C_{nm};n,m\in\mathcal{J}\right\}  $ a sequence of real
numbers. Consider the infinite system of equations
\begin{align*}
u_{m}^{\prime}\left(  t\right)   & =-\lambda_{m}u_{m}\left(  t\right)
+\sum_{n\in\mathcal{J}}C_{nm}u_{n}\left(  t\right)  ,\qquad t\geq0\\
u_{m}\left(  0\right)   & =u_{m}^{0},\qquad m\in\mathcal{J}%
\end{align*}
with given initial condition $\left\{  u_{m}^{0};m\in\mathcal{J}\right\}  $.
We always assume%
\[
\sum_{m\in\mathcal{J}}\left(  u_{m}^{0}\right)  ^{2}<\infty.
\]

\begin{definition}
\label{Def sol}A solution is a sequence $\left\{  u_{m}\left(  \cdot\right)
;m\in\mathcal{J}\right\}  $ of continuous functions on $\left[  0,T\right]  $
such that:

i)%
\[
\sup_{t\in\left[  0,T\right]  }\sum_{m\in\mathcal{J}}u_{m}^{2}\left(
t\right)  +\int_{0}^{T}\sum_{m\in\mathcal{J}}\lambda_{m}u_{m}^{2}\left(
s\right)  ds<\infty
\]

ii)\ the series $\sum_{n\in\mathcal{J}}C_{nm}u_{n}\left(  t\right)  $
converges, for a.e. $t$, to an integrable functions on $\left[  0,T\right]  $ and

iii)\
\[
u_{m}\left(  t\right)  =u_{m}^{0}-\int_{0}^{t}\lambda_{m}u_{m}\left(
s\right)  ds+\int_{0}^{t}\sum_{n\in\mathcal{J}}C_{nm}u_{n}\left(  s\right)
ds
\]
for all $m\in\mathcal{J}$ and $t\in\left[  0,T\right]  $.
\end{definition}

Consider also, for any finite subset $\widetilde{\mathcal{J}}\subset
\mathcal{J}$, the finite system%
\begin{align*}
\widetilde{u}_{m}^{\prime}\left(  t\right)   & =-\lambda_{m}\widetilde{u}%
_{m}\left(  t\right)  +\sum_{n\in\widetilde{\mathcal{J}}}C_{nm}\widetilde{u}%
_{n}\left(  t\right)  ,\qquad t\geq0\\
\widetilde{u}_{m}\left(  0\right)   & =u_{m}^{0},\qquad m\in
\widetilde{\mathcal{J}}%
\end{align*}
The definition of solution for this finite system is obvious and existence and
uniqueness is well known.

\begin{theorem}
Assume that the family $\left\{  C_{nm};n,m\in\mathcal{J}\right\}  $
satisfies, for some constant $C>0$,%
\begin{equation}
\sum_{n,m\in\mathcal{J}}C_{nm}\alpha_{n}\beta_{m}\leq C\left(  \sum
_{n\in\mathcal{J}}\lambda_{n}\alpha_{n}^{2}\right)  ^{1/2}\left(  \sum
_{m\in\mathcal{J}}\beta_{m}^{2}\right)  ^{1/2}\qquad\text{for all sequences
}\left\{  \alpha_{n},\beta_{n};n\in\mathcal{J}\right\}  .\label{assumption}%
\end{equation}
Then there exists a unique solution. Moreover,%
\[
\sup_{t\in\left[  0,T\right]  }\sum_{m\in\widetilde{\mathcal{J}}}\left(
u_{m}\left(  t\right)  -\widetilde{u}_{m}\left(  t\right)  \right)  ^{2}%
+\int_{0}^{T}\sum_{m\in\widetilde{\mathcal{J}}}\lambda_{m}\left(  u_{m}\left(
s\right)  -\widetilde{u}_{m}\left(  s\right)  \right)  ^{2}ds\leq C_{1}%
\int_{0}^{T}\sum_{m\in\widetilde{\mathcal{J}}^{c}}\lambda_{m}u_{m}^{2}\left(
s\right)  ds
\]
for some $C_{1}>0$ independent of $\widetilde{\mathcal{J}}$; where the term
$\int_{0}^{T}\sum_{m\in\widetilde{\mathcal{J}}^{c}}\lambda_{m}u_{m}^{2}\left(
s\right)  ds$ converges to zero as $\widetilde{\mathcal{J}}$ converges to
$\mathcal{J}$.
\end{theorem}

\begin{remark}
Under assumption (\ref{assumption}), given $m_{0}\in\mathcal{J}$ and
$s\in\left[  0,T\right]  $, choose $\alpha_{n}=u_{n}\left(  s\right)  $ and
$\beta_{n}$ equal to zero except for $\beta_{m_{0}}=1$; then
\[
\left\vert \sum_{n\in\mathcal{J}}C_{nm_{0}}u_{n}\left(  s\right)  \right\vert
=\left\vert \sum_{n,m\in\mathcal{J}}C_{nm}u_{n}\left(  s\right)  \beta
_{m}\right\vert \leq C\left(  \sum_{n\in\mathcal{J}}\lambda_{n}u_{n}%
^{2}\left(  s\right)  \right)  ^{1/2}\leq C\left(  1+\sum_{n\in\mathcal{J}%
}\lambda_{n}u_{n}^{2}\left(  s\right)  \right)
\]
hence, in Definition \ref{Def sol}, condition (i) implies (ii).
\end{remark}

\begin{proof}
\textbf{Step 1} (existence and uniqueness). Let $H,V$ be the real separable
Hilbert spaces of sequences $\alpha=\left\{  \alpha_{n};n\in\mathcal{J}%
\right\}  $ such that, respectively $\left\Vert \alpha\right\Vert _{H}%
^{2}:=\sum_{n\in\mathcal{J}}\alpha_{n}^{2}<\infty$, $\left\Vert \alpha
\right\Vert _{V}^{2}:=\sum_{n\in\mathcal{J}}\lambda_{n}\alpha_{n}^{2}<\infty$,
with norms $\left\Vert \alpha\right\Vert _{H}^{2}$ and $\left\Vert
\alpha\right\Vert _{V}^{2}$ respectively; let $\left\langle \cdot
,\cdot\right\rangle _{H}$ denote the inner product in $H$. Since we have
assumed at the beginning that $\left\{  \lambda_{m};m\in\mathcal{J}\right\}  $
diverges to infinity, we have $V\subset H$ and there exists a constant
$C_{H,V}$ such that $\left\Vert \alpha\right\Vert _{H}^{2}\leq C_{H,V}%
\left\Vert \alpha\right\Vert _{V}^{2}$ for all $\alpha\in V$. Let $V^{\prime}$
be the dual space of $V$, with norm $\left\Vert \cdot\right\Vert _{V^{\prime}%
}^{2}$. We identify $H$ with its dual $H^{\prime}$ so that $V\subset H\subset
V^{\prime}$ and denote by $\left\langle \cdot,\cdot\right\rangle $ the dual
pairing between $V$ and $V^{\prime}$, which extends $\left\langle \cdot
,\cdot\right\rangle _{H}$.

Let $a\left(  \cdot,\cdot\right)  :V\times V\rightarrow\mathbb{R}$ be the
bilinear map defined as%
\[
a\left(  \alpha,\beta\right)  =\sum_{n\in\mathcal{J}}\lambda_{n}\alpha
_{n}\beta_{n}-\sum_{n,m\in\mathcal{J}}C_{nm}\alpha_{n}\beta_{m}.
\]
It holds%
\begin{align*}
\left\vert a\left(  \alpha,\beta\right)  \right\vert  & \leq\sum
_{n\in\mathcal{J}}\lambda_{n}\alpha_{n}^{2}+\sum_{n\in\mathcal{J}}\lambda
_{n}\beta_{n}^{2}+C\left(  \sum_{n\in\mathcal{J}}\lambda_{n}\alpha_{n}%
^{2}\right)  ^{1/2}\left(  \sum_{n\in\mathcal{J}}\beta_{n}^{2}\right)
^{1/2}\\
& =\left(  1+C\right)  \left\Vert \alpha\right\Vert _{V}^{2}+\left\Vert
\beta\right\Vert _{V}^{2}+C\left\Vert \beta\right\Vert _{H}^{2}%
\end{align*}
hence $a\left(  \cdot,\cdot\right)  $ is well defined and continuous on
$V\times V$. Moreover, since%
\[
C\left(  \sum_{n\in\mathcal{J}}\lambda_{n}\alpha_{n}^{2}\right)  ^{1/2}\left(
\sum_{n\in\mathcal{J}}\beta_{n}^{2}\right)  ^{1/2}\leq\frac{1}{2}\sum
_{n\in\mathcal{J}}\lambda_{n}\alpha_{n}^{2}+2C^{2}\sum_{n\in\mathcal{J}}%
\beta_{n}^{2}%
\]
we get%
\[
a\left(  \alpha,\alpha\right)  =\sum_{n\in\mathcal{J}}\lambda_{n}\alpha
_{n}^{2}-\sum_{n,m\in\mathcal{J}}C_{nm}\alpha_{n}\alpha_{m}\geq\frac{1}{2}%
\sum_{n\in\mathcal{J}}\lambda_{n}\alpha_{n}^{2}-2C^{2}\left\Vert
\alpha\right\Vert _{H}^{2}%
\]
hence $a\left(  \cdot,\cdot\right)  $ is coercive on $V\times V$. Consider the
equation
\[
\left\langle u\left(  t\right)  ,\phi\right\rangle _{H}+\int_{0}^{t}a\left(
u\left(  s\right)  ,\phi\right)  ds=\left\langle u^{0},\phi\right\rangle
_{H}+\int_{0}^{t}\left\langle f\left(  s\right)  ,\phi\right\rangle ds
\]
with $\phi\in V$, $u^{0}\in H$, $f\in L^{2}\left(  0,T;H\right)  $ (one can
treat also $f\in L^{2}\left(  0,T;V^{\prime}\right)  $ but this is not
important here). By solution we mean a function $u\in L^{\infty}\left(
0,T;H\right)  \cap L^{2}\left(  0,T;V\right)  $ which satisfies this equation
for all $\phi\in V$ and all $t\in\left[  0,T\right]  $. By a well known
theorem (see \cite{Lions}), there exists a unique solution of this equation,
with%
\[
\sup_{t\in\left[  0,T\right]  }\left\Vert u\left(  t\right)  \right\Vert
_{H}^{2}+\int_{0}^{T}\left\Vert u\left(  s\right)  \right\Vert _{V}%
^{2}ds<\infty.
\]
This proves existence and uniqueness of a solution of the infinite system
above, in the sense of Definition \ref{Def sol}.

\textbf{Step 2} (convergence) Let us prove the estimate between the finite and
infinite system. We have%
\[
u_{m}\left(  t\right)  =u_{m}^{0}-\int_{0}^{t}\lambda_{m}u_{m}\left(
s\right)  ds+\int_{0}^{t}\sum_{n\in\widetilde{\mathcal{J}}}C_{nm}u_{n}\left(
s\right)  ds+\int_{0}^{t}R_{m}^{\widetilde{\mathcal{J}}}\left(  s\right)  ds
\]
where $R_{m}^{\widetilde{\mathcal{J}}}\left(  s\right)  =\sum_{n\in
\widetilde{\mathcal{J}}^{c}}C_{nm}u_{n}\left(  s\right)  $;\ we know that
$R_{m}^{\widetilde{\mathcal{J}}}$ is an integrable function, by definition of
solution. Then, for the new variable $v_{m}\left(  t\right)  :=u_{m}\left(
t\right)  -\widetilde{u}_{m}\left(  t\right)  $ we have
\[
v_{m}\left(  t\right)  =-\int_{0}^{t}\lambda_{m}v_{m}\left(  s\right)
ds+\int_{0}^{t}\sum_{n\in\widetilde{\mathcal{J}}}C_{nm}v_{n}\left(  s\right)
ds+\int_{0}^{t}R_{m}^{\widetilde{\mathcal{J}}}\left(  s\right)  ds.
\]
It follows that the family $\left\{  v_{m};m\in\widetilde{\mathcal{J}%
}\right\}  $ satisfies the finite system
\begin{align*}
v_{m}^{\prime}\left(  t\right)    & =-\lambda_{m}v_{m}\left(  t\right)
+\sum_{n\in\widetilde{\mathcal{J}}}C_{nm}v_{n}\left(  t\right)  +R_{m}%
^{\widetilde{\mathcal{J}}}\left(  t\right)  ,\qquad t\geq0\\
v_{m}\left(  0\right)    & =0,\qquad m\in\widetilde{\mathcal{J}}.
\end{align*}
We have%
\begin{align*}
\sum_{m\in\widetilde{\mathcal{J}}}v_{m}R_{m}^{\widetilde{\mathcal{J}}}  &
=\sum_{m\in\widetilde{\mathcal{J}}}\sum_{n\in\widetilde{\mathcal{J}}^{c}%
}C_{nm}u_{n}v_{m}\leq C\left(  \sum_{n\in\widetilde{\mathcal{J}}^{c}}%
\lambda_{n}u_{n}^{2}\right)  ^{1/2}\left(  \sum_{n\in\widetilde{\mathcal{J}}%
}v_{n}^{2}\right)  ^{1/2}\\
& \leq C^{2}\sum_{n\in\widetilde{\mathcal{J}}}v_{n}^{2}+\sum_{n\in
\widetilde{\mathcal{J}}^{c}}\lambda_{n}u_{n}^{2}%
\end{align*}
and thus%
\begin{align*}
\frac{1}{2}\frac{d}{dt}\sum_{m\in\widetilde{\mathcal{J}}}v_{m}^{2}+\sum
_{m\in\widetilde{\mathcal{J}}}\lambda_{m}v_{m}^{2}  & =\sum_{n,m\in
\widetilde{\mathcal{J}}}C_{nm}v_{n}v_{m}+\sum_{m\in\widetilde{\mathcal{J}}%
}v_{m}R_{m}^{\widetilde{\mathcal{J}}}\\
& \leq C\left(  \sum_{n\in\mathcal{J}}\lambda_{n}v_{n}^{2}\right)
^{1/2}\left(  \sum_{n\in\mathcal{J}}v_{n}^{2}\right)  ^{1/2}+C^{2}\sum
_{n\in\widetilde{\mathcal{J}}}v_{n}^{2}+\sum_{n\in\widetilde{\mathcal{J}}^{c}%
}\lambda_{n}u_{n}^{2}\\
& \leq\frac{1}{2}\sum_{m\in\widetilde{\mathcal{J}}}\lambda_{m}v_{m}^{2}%
+3C^{2}\sum_{m\in\widetilde{\mathcal{J}}}v_{m}^{2}+\sum_{n\in
\widetilde{\mathcal{J}}^{c}}\lambda_{n}u_{n}^{2}%
\end{align*}
hence (renaming the constant $C$)%
\[
\frac{1}{2}\frac{d}{dt}\sum_{m\in\widetilde{\mathcal{J}}}v_{m}^{2}+\frac{1}%
{2}\sum_{m\in\widetilde{\mathcal{J}}}\lambda_{m}v_{m}^{2}\leq3C^{2}\sum
_{m\in\widetilde{\mathcal{J}}}v_{m}^{2}+\sum_{n\in\widetilde{\mathcal{J}}^{c}%
}\lambda_{n}u_{n}^{2}%
\]
which, by Gronwall lemma, easily implies that there exists a constant
$C_{1}>0$, independent of the finite subset $\widetilde{\mathcal{J}}$, such
that%
\[
\sup_{t\in\left[  0,T\right]  }\sum_{m\in\widetilde{\mathcal{J}}}v_{m}%
^{2}\left(  t\right)  +\int_{0}^{T}\sum_{m\in\widetilde{\mathcal{J}}}%
\lambda_{m}v_{m}^{2}\left(  s\right)  ds\leq C_{1}\int_{0}^{T}\sum
_{n\in\widetilde{\mathcal{J}}^{c}}\lambda_{n}u_{n}^{2}\left(  s\right)  ds.
\]

The proof is complete.
\end{proof}

\begin{proposition}
Let $B:\mathcal{H}\rightarrow\mathcal{H}$ be bounded measurable and let
$C_{nm}$ be given by
\[
C_{nm}=\int_{\mathcal{H}}\left\langle B\left(  x\right)  ,D_{x}H_{n}\left(
x\right)  \right\rangle _{\mathcal{H}}H_{m}\left(  x\right)  \mu\left(
dx\right)  .
\]
If%
\begin{align}
 \int_{\mathcal{H}}\left\vert D_{x}\varphi\left(  x\right)  \right\vert
_{\mathcal{H}}^{2}\mu\left(  dx\right)  \leq 2 \sum_{n\in\mathcal{J}}\lambda
_{n}\varphi_{n}^{2}\label{s6.2.1}
\end{align}
for every function $\varphi\left(  x\right)  $ of the form $\varphi\left(
x\right)  =\sum_{n\in\mathcal{J}}\varphi_{n}H_{n}\left(  x\right)  $, then
condition (\ref{assumption}) holds true.
\end{proposition}

\begin{proof}
Given two sequences $\left\{  \alpha_{n},\beta_{n};n\in\mathcal{J}\right\}  $,
setting
\[
\varphi\left(  x\right)  =\sum_{n\in\mathcal{J}}\alpha_{n}H_{n}\left(
x\right)  ,\qquad\psi\left(  x\right)  =\sum_{m\in\mathcal{J}}\beta_{m}%
H_{m}\left(  x\right)
\]
one simply has%
\begin{align*}
\sum_{n,m\in\mathcal{J}}C_{nm}\alpha_{n}\beta_{m}  & =\int_{\mathcal{H}%
}\left\langle B\left(  x\right)  ,D_{x}\varphi\left(  x\right)  \right\rangle
_{\mathcal{H}}\psi\left(  x\right)  \mu\left(  dx\right)  \\
& \leq\left\Vert B\right\Vert _{\infty}\int_{\mathcal{H}}\left\vert
D_{x}\varphi\left(  x\right)  \right\vert _{\mathcal{H}}\left\vert \psi\left(
x\right)  \right\vert \mu\left(  dx\right)  \\
& \leq\left\Vert B\right\Vert _{\infty}\left(  \int_{\mathcal{H}}\left\vert
D_{x}\varphi\left(  x\right)  \right\vert _{\mathcal{H}}^{2}\mu\left(
dx\right)  \right)  ^{1/2}\left(  \int_{\mathcal{H}}\left\vert \psi\left(
x\right)  \right\vert ^{2}\mu\left(  dx\right)  \right)  ^{1/2}\\
& \leq\left\Vert B\right\Vert _{\infty}\left(2  \sum_{n\in\mathcal{J}}%
\lambda_{n}\alpha_{n}^{2}\right)  ^{1/2}\left(  \sum_{m\in\mathcal{J}}%
\beta_{m}^{2}\right)  ^{1/2}.
\end{align*}

\end{proof}

 Now, we will prove that \eqref{s6.2.1} is satisfied in our case. Assume the conditions in Lemma \ref{s1.le2} holds. Then,
 for any $\Phi,\Psi\in \mathcal{S}(\mathbb{H})$\footnote{Recall that  $\mathcal{S}(\mathbb{H})$ is the set of all cylinder functionals on 
 $\mathcal{H}$}, the following Green's formula holds (for a proof see Lemma 4.4 in \cite{liu} for instance)
 \begin{align*}
-\frac{1}{2} \int_{\mathcal{H}} \langle QD_x\Phi,D_x\Psi \rangle_{\mathcal{H}} \mu(dx)=
\int_{\mathcal{H}} (\mathcal{A}_0\Phi)\Psi  \mu(dx)= \int_{\mathcal{H}} \Phi(\mathcal{A}_0\Psi)  \mu(dx).
 \end{align*}
 By taking $\Psi=\Phi=\varphi$ and $Q=Id$ we have
 \begin{align*}
  \int_{\mathcal{H}}  |D_x\varphi|_{\mathcal{H}}^2  \mu(dx)=
 \int_{\mathcal{H}} \langle D_x\varphi,D_x\varphi \rangle_{\mathcal{H}} \mu(dx)=
 -2\int_{\mathcal{H}} (\mathcal{A}_0\varphi)\varphi  \mu(dx)
 \end{align*}
 If $\varphi\left(x\right)  =\sum_{n\in\mathcal{J}}\varphi_{n}H_{n}\left(  x\right)  $, then 
 \begin{align*}
 -\int_{\mathcal{H}} (\mathcal{A}_0\varphi)\varphi  \mu(dx)&= 
 \int_{\mathcal{H}} \Bigg(-\mathcal{A}_0 \sum_{n\in\mathcal{J}}\varphi_{n}H_{n}\left(  x\right) \Bigg)
 \sum_{m\in\mathcal{J}}\varphi_{m}H_{m}\left(  x\right)  \mu(dx)\\
 &= \sum_{m\in\mathcal{J}} \int_{\mathcal{H}} \Bigg(\sum_{n\in\mathcal{J}}\varphi_{n}\big[-\mathcal{A}_0 H_{n}\left(  x\right)\big] \Bigg)
 \varphi_{m}H_{m}\left(  x\right)  \mu(dx)\\
 & = \sum_{m\in\mathcal{J}} \int_{\mathcal{H}} \sum_{n\in\mathcal{J}}\varphi_{n}\lambda_n H_{n}\left(  x\right) 
 \varphi_{m}H_{m}\left(  x\right)  \mu(dx)\\
  & = \sum_{m\in\mathcal{J}} \sum_{n\in\mathcal{J}}\varphi_{n} \varphi_{m} \lambda_n  \int_{\mathcal{H}} H_{n}\left(  x\right) 
 H_{m}\left(  x\right)  \mu(dx)\\
 &= \sum_{n\in\mathcal{J}}\lambda_n \varphi_{n}^2.
 \end{align*}
Where in the last step we will use that $ H_{n}\left(  x\right)$ is an orthonormal basis for $\mathcal{H} $. Then, we have
\begin{align*}
  \int_{\mathcal{H}}  |D_x\varphi|_{\mathcal{H}}^2  \mu(dx)=
  -2\int_{\mathcal{H}} (\mathcal{A}_0\varphi)\varphi  \mu(dx)
  = 2\sum_{n\in\mathcal{J}}\lambda_n\varphi_{n}^2 .
 \end{align*}

\section{Numerical Results}\label{AM-sect}

 \subsection{Algorithm description}
 
In this subsection we describe the algorithm we follow to get the simulations for the Kolmogorov equations associated with 
three stochastic partial differential equations whose results we show in next subsections.

\begin{tcolorbox}
 
\begin{enumerate}
 \item Choose the algorithm's parameters:
 \begin{itemize}
  \item[a)] The space $\mathcal{H}$ where the SPDE will be defined.
  \item[b)] The operator $A$ and its eigenfunctions $\lambda_k $ and eigenvalues $e_k(\cdot)$.
  \item[c)] The functional $u_0:\mathcal{H}\rightarrow \IR$.
  \item[d)] $N,M$ and then fix the set $J^{N,M}$.
  \item[e)] The time step $\Delta t$ and $\Delta x$ in the physical space.
 \end{itemize}
\item Compute the quantities $\bar{C}_{\bm{n},\bm{m}}$, for each $\bm{n},\bm{m}\in J^{N,M} $, to approximate \eqref{C-NM}.
 \item Set the finite system of coupled ordinary differential equation \eqref{fin-sys}
  \item Rewriting the system  \eqref{fin-sys} as a matrix differential equations and by solving it
  numerically we obtain, up to a set of constants, the time-functions $u_{\bm{n}}(t)$, for each $\bm{n}\in J^{N,M}$.
\item By using the functional $u_0$ the constants in the last step are fixed.
\item We then define the space-time approximation for the Kolmogorov equation as 
\begin{align*}
u_N(t,x)=\sum_{j=1}^N u_j(t)H_j(x)\approx \sum_{j\ge 1} u_j(t)H_j(x) = u(t,x)
\end{align*}
\end{enumerate}

\end{tcolorbox}

\begin{remark}

 \begin{itemize}
  \item  Given the operator $A$, we choose its eigenvalues as the basis for the Hilbert space $\mathcal{H}$ and we have 
  to find its eigenvalues $\lambda_k $.
\item The choice of the functional $u_0$ will change the way we determine the initial condition of the Kolmogorov equation, then 
it will necessary to adapt the method for each $u_0$.
\item the quantities $\bar{C}_{\bm{n},\bm{m}}$ are those that require more computing resources because we have to compute 
and approximate several integrals for each $\bm{n},\bm{m}\in J^{N,M}$. In our examples these quantities are given by
the expressions \eqref{C1heat}, \eqref{C1fisher} and \eqref{C1burger}.
 \end{itemize}
 
\end{remark}

\subsection{Stochastic Heat equation in an interval}
As a first application consider the stochastic diffusion in dimension 1.

Let $\mathcal{H}=L^2([0,1])$, $Q=Id$, and $A$ be given by $Ax=\nu\bigtriangleup_\xi x$, $x\in D(A)$ with 
$D(A)=H^2(0,1)\cap H_0^1(0,1)$ (where $H^2(0,1)$ is the Sobolev spaces and $H_0^1(0,1)$ is the subspace of $H^1(0,1)$ 
of all functions vanishing at $0,1$). 

Consider the heat equation in $[0,1]$

\begin{align}
 \frac{\partial X(t,\xi)}{\partial t}&= \nu\frac{\partial^2 X(t,\xi)}{\partial \xi^2}+f(\xi)+ 
 \frac{\partial^2 W}{\partial t\partial \xi},\quad \xi\in[0,1] \label{dif1}\\
 X(t,\xi)\mid_{t=0}&=X_0(\xi),\quad X_0\in \mathcal{H},\nonumber\\
 X(t,\xi)&=0,\quad t\ge 0, \xi=0,1,\nonumber
\end{align}
where  $t\in [0,T]$, $f(\xi)=\xi^3$, $X_0(\xi)=\sin(\pi x)$. $W$ is a cylindrical Wiener process 
on $\mathcal{H}$, associated to a stochastic basis $(\Omega,\mathcal{F},\Pb,(\mathcal{F}_t)_{t\ge 0})$. $\nu$ denotes the
thermal diffusivity.

The complete orthonormal system of eigenfunctions $e_k$ is defined as
\begin{align*}
 e_k(\xi)=\sqrt{2}\sin(k\pi\xi),\qquad \xi\in [0,1],\quad k\in\IN.
\end{align*}

A is self-adjoint negative operator and $Ae_k=-\nu k^2\pi^2 e_k$, $k\in\IN$.

We rewrite the equation \eqref{dif1} as an abstract differential equation on $\mathcal{H}$. Set $B=f$, then
\begin{align*}
 dX&=[AX+B(X)]dt + dW_t,\\
 X(0)&=x,\quad x\in \mathcal{H}
\end{align*}

Define $ u(t,x)=\E\big[ u_0(X_t^x)\big]$ and then $u(t,x)$ satisfies the Kolmogorov equation
\begin{equation*}
\frac{\partial u}{\partial t}= \frac{1}{2}Tr(QD^2u)+ \langle Ax, Du \rangle_\mathcal{H} + \langle B(x),Du \rangle_\mathcal{H},\qquad x\in D(A).
\end{equation*}

\vspace*{-0.1cm} We will consider two cases :  
 \begin{align*}
  u_0^{\xi_0}(g)&:= g(\xi_0).\qquad \mbox{for fixed } \xi_0\in(0,1)\\
 &\mbox{\hspace*{-3cm} and }\\
 \vspace*{-0.1cm}  u_0(g)&:=\int_0^1 g(\xi) d\xi.
 \end{align*}

As before we write the solution as 
\begin{equation}
u(t,x) =  \sum_{\bm{n}\in \mathcal{J}} u_{\bm{n}}(t) H_{\bm{n}}(x), \qquad x\in\mathcal{H},\quad t\in[0,T],
\end{equation}
where $u_{\bm{n}}:[0,T]\mapsto \IR$ and $H_{\bm{n}}(x)$ are the Hermite functionals.
Following the last procedure  we set the infinite system of coupled ordinary differential equations.
\begin{equation}
   \dot{u}_{\bm{m}}(t) =   -u_{\bm{m}}(t) \lambda_{\bm{m}} 
  + \sum_{\bm{n}\in \mathcal{J}} u_{\bm{n}}(t) C_{\bm{n},\bm{m}},\qquad \bm{n},\bm{m}\in\mathcal{J}\label{inf-sys-dif}
\end{equation}
where $ C_{\bm{n},\bm{m}}$ is given by
\begin{equation}
C_{\bm{n},\bm{m}}:=\int_\mathcal{H} \big\langle B(x), D_x H_{\bm{n}}(x) \big\rangle_\mathcal{H} H_{\bm{m}}(x) \mu(dx).\label{C-NM1}
 \end{equation}

The numerical method for this case is applied now. We have that $\Lambda = \tfrac{1}{2}(-A)^{-1}$ have eigenvalues
$1/(2\nu\pi^2|k|^2) $, then the operator 
$\Lambda^{-1}$ is well-defined and have eigenvalues $2\nu\pi^2|k|^2 $, and $\Lambda^{-\tfrac{1}{2}}$ can
also be befined having eigenvalues $ \sqrt{2\nu}\pi|k|$, then 
\[
 \Big\langle B(x), \Lambda^{-\tfrac{1}{2}} e_k\Big\rangle_{L^2([0,1])} = \sqrt{2\nu}\pi|k|\big\langle f,
 e_k\big\rangle_{L^2([0,1])}.
\]

Notice that   $H_{\bm{n}}=\prod_{\alpha} P_{n_\alpha}(\xi_\alpha)$,  $H_{\bm{m}}=\prod_{\alpha} P_{m_\alpha}(\xi_\alpha)$ and $
P_{m_k}^{'} \big(\xi_k\big) =m_k^{1/2} P_{m_k-1}\big(\xi_k\big)$. Then, we rewrite 
$ C_{\bm{n},\bm{m}} $ as follows.

\begin{align*}
 C_{\bm{n},\bm{m}} &=\sum_{k=1}^\infty\int_\mathcal{H}  \Big\langle B(x),  \Lambda^{-\tfrac{1}{2}} e_k\Big\rangle_\mathcal{H} 
 \prod_{\stackrel{i=1}{i\ne k}}^\infty P_{n_i}\big(\langle x,  \Lambda^{-\tfrac{1}{2}} e_i\rangle_\mathcal{H}\big) P_{n_k}^{'} 
 \big(\langle x,\Lambda^{-\tfrac{1}{2}} e_k\rangle_\mathcal{H}\big) H_{\bm{m}}(x) \mu(dx)\\
 &=\sum_{k=1}^\infty\int_\mathcal{H}  \lambda_k \big\langle f,  e_k\big\rangle_\mathcal{H}
 P_{m_k}\big(\langle x,  \Lambda^{-\tfrac{1}{2}} e_k\rangle_\mathcal{H}\big)
  P_{n_k}^{'}  \big(\langle x,\Lambda^{-\tfrac{1}{2}} e_k\rangle_\mathcal{H}\big)\\
  &\qquad\qquad\times\prod_{\stackrel{i=1}{i\ne k}}^\infty P_{n_i}\big(\langle x,  
  \Lambda^{-\tfrac{1}{2}} e_i\rangle_\mathcal{H}\big) P_{m_i}\big(\langle x,  \Lambda^{-\tfrac{1}{2}} e_i\rangle_\mathcal{H}\big)
  \mu(dx).
 \end{align*}
 Writing the measure $\mu(dx)$ in the direction $e_k$  as 
 $\mu(dx)e_k=\tfrac{1}{\lambda_k}\mu\Big(d\big(\langle x,  \Lambda^{-\tfrac{1}{2}} e_k\rangle_\mathcal{H}\big)\Big)=
 \tfrac{1}{\lambda_k}\mu(d\xi_k)$ with $\xi_k=\langle x,  \Lambda^{-\tfrac{1}{2}} e_k\rangle_\mathcal{H}$, then we approximate
 $ C_{\bm{n},\bm{m}}$ as

\begin{align*}
 C_{\bm{n},\bm{m}} &=\sum_{k=1}^{\infty} \lambda_k \int_0^{1} f(\xi) e_k(\xi) d\xi 
\int_{\IR} n_k^{1/2} P_{m_k}\big(\xi_k\big) P_{n_k-1}\big(\xi_k\big)\frac{1}{\lambda_k}\mu(d\xi_k)\\ 
  &\qquad\quad\times \int_{\IR^{\IN}} \prod_{\stackrel{i=1}{i\ne k}}^{\infty} P_{n_i}(\xi_i)P_{m_i}(\xi_i) \frac{1}{\lambda_i} \mu(d\xi_i)   \\
&\approx \sum_{k=1}^{M}  \int_0^{1} f(\xi) e_k(\xi) d\xi 
\int_{\IR} n_k^{1/2} P_{m_k}\big(\xi_k\big) P_{n_k-1}\big(\xi_k\big)\mu(d\xi_k)\\ 
  &\qquad\quad\times \int_{\IR^{M-1}} \prod_{\stackrel{i=1}{i\ne k}}^{M} P_{n_i}(\xi_i)P_{m_i}(\xi_i)\frac{1}{\lambda_i} \mu(d\xi_i)   \\
&=  \sum_{k=1}^{M}  n_k^{1/2} \int_0^{1} f(\xi) e_k(\xi)  d\xi 
\int_{\IR}  P_{m_k}\big(\xi_k\big) P_{n_k-1}\big(\xi_k\big)\mu(d\xi_k)\\ 
  &\qquad\quad\times \prod_{\stackrel{i=1}{i\ne k}}^{M}  \frac{1}{\lambda_i}\int_{\IR}  P_{n_i}(\xi_i)P_{m_i}(\xi_i) \mu(d\xi_i)    
 \end{align*}

For $N_1\in\IN$ define  the set $S_{N_1}=\{\bm{n}_1,\bm{n}_2,\ldots,\bm{n}_{N_1}: \bm{n}_i\in J^{M,N}, i=1,\ldots,N_1 \}$. Moreover, 
for $\bm{n},\bm{m}\in S_{M}$ define
\begin{align}
\bar C_{\bm{n},\bm{m}}&:= \sum_{k=1}^{M} \sqrt{2\nu}\pi|k| n_k^{1/2} \int_0^{1} f(\xi) e_k(\xi) d\xi 
\int_{\IR}  P_{m_k}\big(\xi_k\big) P_{n_k-1}\big(\xi_k\big)\mu(d\xi_k)\nonumber\\ 
  &\qquad\quad\times \prod_{\stackrel{i=1}{i\ne k}}^{M} \int_{\IR}  P_{n_i}(\xi_i)P_{m_i}(\xi_i) \mu(d\xi_i), \label{C1heat}   
\end{align}
and the finite system of ordinary differential equations: 
\begin{equation}
   \dot{u}_{\bm{m}}(t) =   -u_{\bm{m}}(t) \lambda_{\bm{m}} 
  + \sum_{\bm{n}\in S_{M}} u_{\bm{n}}(t) \bar C_{\bm{n},\bm{m}},\qquad \mbox{ for each }\bm{m}\in S_{M} \mbox{ and } \bm{n} \in S_{M}.
  \label{fin-sys-dif3}
\end{equation}
Then \eqref{fin-sys-dif3} approximates to the infinite system of ordinary differential equations \eqref{inf-sys-dif} when 
$N,M \rightarrow \infty $. We use the system  \eqref{fin-sys-dif3} to approximate the solution of the FPK equation 
associated with the Diffusion equation.
 
We need to evaluate the integrals and do the finite sum on $k$, to do this we use a Gauss-Hermite quadrature to approximate the value 
of the integrals 
\[
 \int_0^{1} f(\xi) e_k(\xi) d\xi , \qquad  \int_{\IR} P_{m_k}\big(\xi_k\big) P_{n_k-1}\big(\xi_k\big)\mu(d\xi_k)   , \qquad 
 \int_{\IR}  P_{n_i}(\xi_i)P_{m_i}(\xi_i) \mu(d\xi_i).
 \]
 
%  To generate the values of the $N$-tuple $\bm{n},\bm{m}$ we first produce randomly a vector of $N$ random variables Bernoulli and then 
% we multiply each entry by a uniform($N$). {\color{red} This part is a crucial part of the simulation: recall the definition $J^{M,N}$,  according to previous simulations if $N>80$ and if in each $\bm{m}_i$ the number of index's, different from zero, is grater than $6$ then the constant $C_{\bm{n},\bm{m}} $ could be $\pm \infty$ and then we cannot have a solution for the system \eqref{MDE}}. Then, we choose $N=60$, and $M=20$ which give us that all the constant's $C_{\bm{n},\bm{m}} $  are finite and then we are able to obtain solutions to the system \eqref{MDE}.

When the constans $C_{\bm{n},\bm{m}}$  are fixed we solve the Matrix Differential equation \eqref{MDE}:
\begin{align}
 \dot{U}^M(t)=A U^M(t).\label{MDE1}
\end{align}
with
\[ 
A=\left( \begin{array}{ccccc}
-\lambda_{1}+ C_{1,1} & C_{2,1} & \cdots & C_{M-1,1} & C_{M,1} \\
 C_{1,2} & -\lambda_{2}+C_{2,2} & \cdots & C_{M-1,2} & C_{M,2} \\
\vdots &\vdots&\ddots&\vdots&\vdots \\
C_{1,M-1} & C_{2,M-1} & \cdots & -\lambda_{M-1}+ C_{M-1,M-1} & C_{M,M-1} \\
C_{1,M} & C_{2,M} & \cdots & C_{M-1,M} & -\lambda_{M}+ C_{M,M} \\
\end{array} 
\right),
\] 
 $\lambda_{i}=\lambda_{\bm{m}_i}$ and $C_{i,j}=C_{\bm{n}_i,\bm{m}_j}$ for $1\le i,j\le M$, and 
\begin{align*}
 U^M(t)&=\big( u_{\bm{m}_1}(t),u_{\bm{m}_2}(t),\ldots,u_{\bm{m}_M}(t) \big)^T\\
 \dot{U}^M(t)&=\big(\dot{u}_{\bm{m}_1}(t),\dot{u}_{\bm{m}_2}(t),\ldots,\dot{u}_{\bm{m}_M}(t) \big)^T
 \end{align*}
 
From this we get the general solution of \eqref{MDE1} is given by 
 \begin{equation}
 \left( \begin{array}{c}
 u_1(t)\\
u_2(t) \\
\vdots \\
u_{M-1}(t) \\
u_M(t)\\
\end{array} 
\right)=\left(\mathbf{V}_1 \quad \mathbf{V}_2 \quad\cdots \quad\mathbf{V}_{M-1} \quad \mathbf{V}_M \right)
 \left( \begin{array}{c}
 c_1e^{\lambda_1 t }\\
 c_2e^{\lambda_2 t }\\
\vdots \\
 c_{M-1}e^{\lambda_{M-1} t }\\
 c_Me^{\lambda_M t }\\
\end{array} 
\right)
\label{ME-BC1}
\end{equation}
 where $\mathbf{V}_i$ and $\lambda_i$ are the eigenvector and eigenvalue of the matrix $A$. It remains to fix the 
 set of constants $\{c_i,  1\le i\le M  \}$ which are  determined by using the initial conditions given in subsection \ref{ICsect}.
 
 \subsubsection*{Initial Conditions} 

We define the set points in the set $[0,1]$ as $\{\xi_i\}$, $i=0,\ldots,P$, such that $\xi_0=0$ and $\xi_{P}=1$. Then by using 
\eqref{ME-BC3} we  fix the values of the constants $c_i$

\begin{equation}
\left( \begin{array}{c}
 c_1\\
 c_2\\
\vdots \\
 c_M\\
\end{array} 
\right)
 =\left[\left( \begin{array}{cccc}
H_{1}(x_0)& H_{2}(x_0) & \cdots  & H_{m}(x_0) \\
H_{1}(x_1)& H_{2}(x_1) & \cdots  & H_{m}(x_1) \\
\vdots &\vdots&\ddots&\vdots \\
H_{1}(x_{P-1})& H_{2}(x_{P-1}) & \cdots  & H_{m}(x_{P-1}) \\
H_{1}(x_{P})& H_{2}(x_{P}) & \cdots & H_{m}(x_{P}) \\
\end{array} 
\right)\left(
\begin{array}{c}\mathbf{V}_1 \\
 \mathbf{V}_2 \\
 \vdots \\
 \mathbf{V}_M \\
 \end{array}\right)^T\right]^{-1}
 \left( \begin{array}{c}
 X_0(\xi_0)\\
X_0(\xi_1) \\
\vdots \\
X_0(\xi_{P})\\
\end{array} 
\right),
\label{ME-BC5}
\end{equation}
 
 With this, we have now completed the process to build the approximation for the solution.

\subsubsection{Deterministic equation associated with the stochastic diffusion \eqref{dif1}}\label{det1}
 
 Set 
 \begin{align*}
  y(t,\xi)=\E\big[X_t(\xi)\big]
 \end{align*}
then, $y(t,\xi)$ solves the differential equation 
\begin{align}
 \frac{\partial y}{\partial t} &=  \nu\frac{\partial^2 y}{\partial \xi^2}+f\label{dif2}\\
 y\big|_{t=0}&= \E(X_0).\nonumber
\end{align}

We solve numerically this equation by using the Matlab library {\it pdepe} and we compare our results by using the spectral method with the one
obtained with the pdepe Matlab library.

% 
% \begin{figure}[!ht]                 
% \includegraphics[scale=0.35]{Diffusion_pdepe_nu=0_1.jpg}
% \caption{Diffusion equation with $\nu=0.01$ obtained by using the pdepe library.}
% \end{figure}\vspace*{0.1cm}
%  

\subsubsection*{Results on the simulation }

We have the following graphs of simulations using this method with differents values of $J^{N,M}$, $N=7,8$. We make a comparison
with the solution of the deterministic equation, as was described in subsection \ref{det1}, by using the matlab library {\it pdepe}.

First we show the result on the simulation for the evaluation functional. The second group of graphs shows the simulation for the 
second functional. The results were obtained with the coefficient $\nu=0.1$. 
\vspace*{0.1cm}

\begin{figure}[H]        
\includegraphics[width = 1.95in]{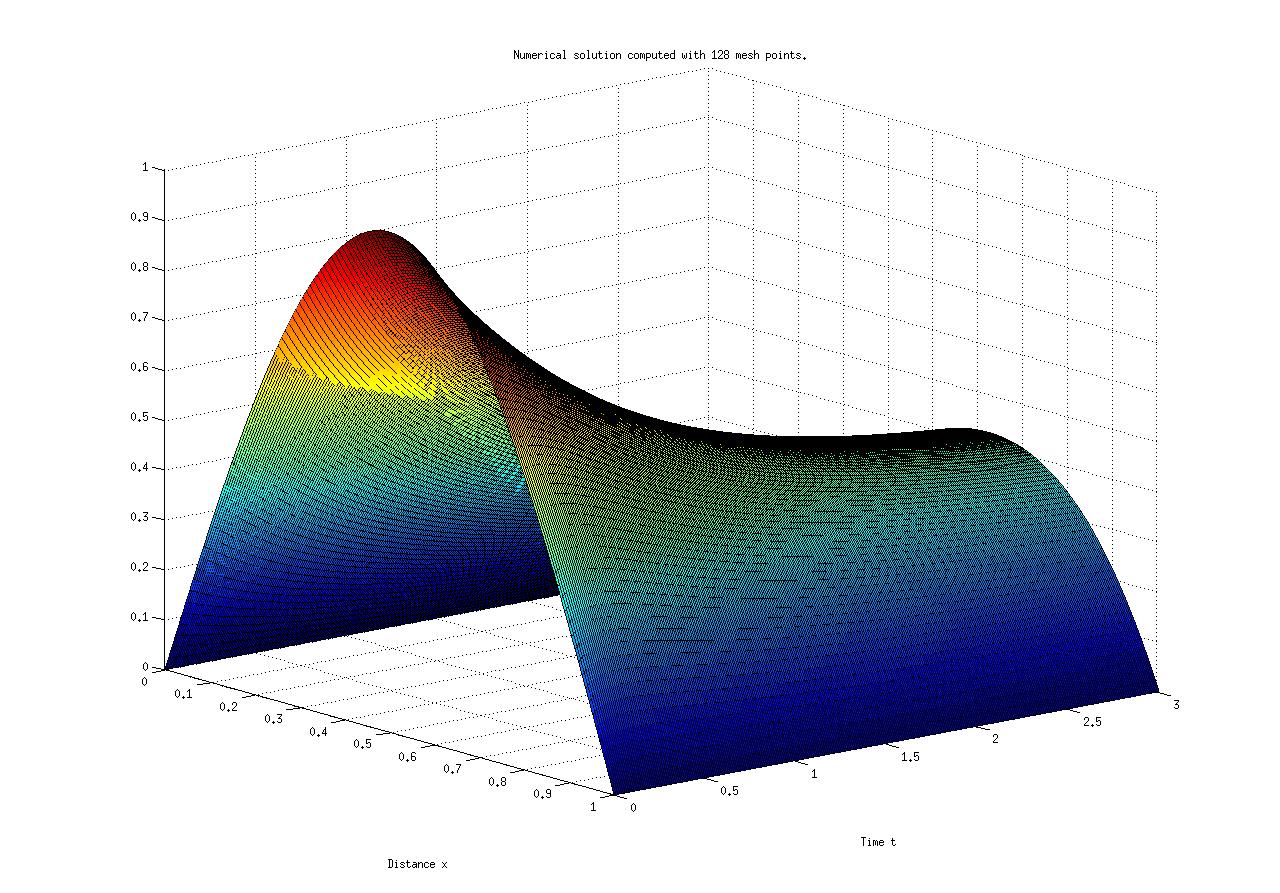}
\includegraphics[width = 1.95in]{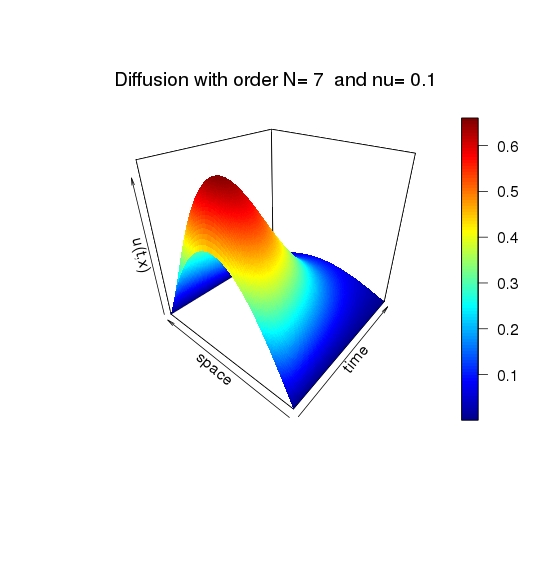} 
\includegraphics[width = 1.95in]{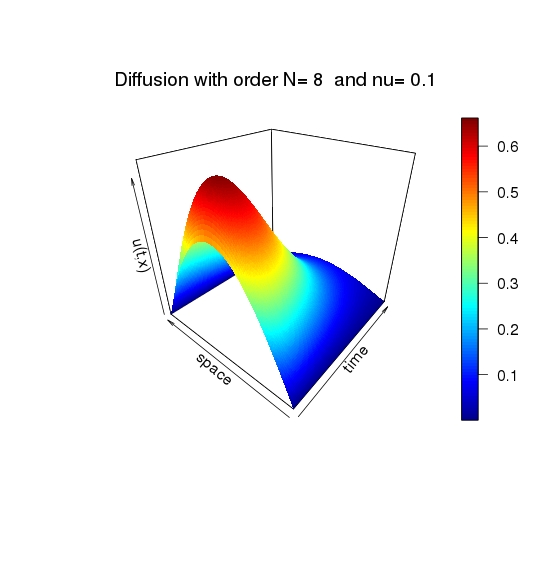} 
\caption{Simulations for the Diffusion equation with the spectral method, for $N=7,8$ 
and $\nu=0.1$ with $u_0^{\xi_0}(g)=g(\xi_0)$.}
\label{graph-simu_Difnu1}
\end{figure}\vspace*{0.1cm}

\begin{figure}[H]        
\includegraphics[width = 1.95in]{Diffusion_pdepe_nu=0_1.jpg}
\includegraphics[width = 1.95in]{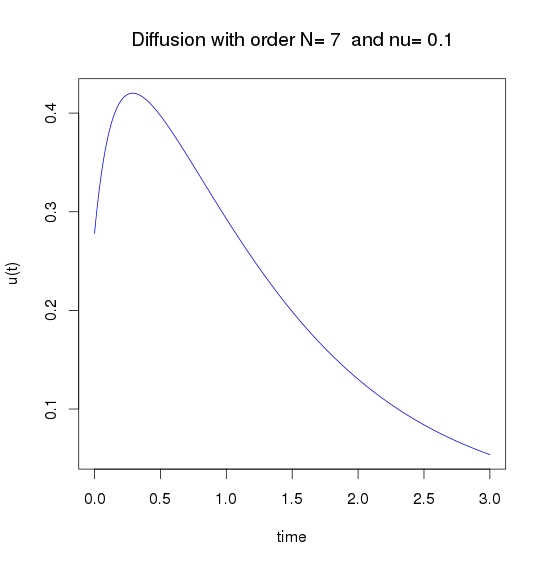} 
\includegraphics[width = 1.95in]{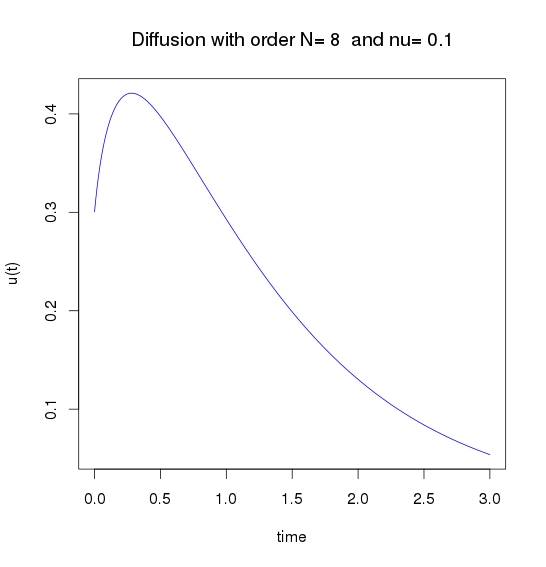} 
\caption{Simulations for the Diffusion equation with the spectral method, for $N=7,8$, $\nu=0.1$
and $u_0(g)=\int_0^1 g(\xi) d\xi)$..}
\label{graph-simu_Difnu3}
\end{figure}\vspace*{0.1cm}

\subsection{Stochastic Fisher-KPP Equation in an interval}

Set $\mathcal{H}=L^2(0,1)$. We consider the stochastic Fisher-KPP equation in the interval $[0,1]$: 
\begin{align}
 dX(t,\xi)&=\Big[\nu\partial_{\xi}^2X(t,\xi)+X(t,\xi)(1-X(t,\xi)) \Big]dt+dW_t(t,\xi),\quad t>0,\quad \xi\in(0,1)\label{fisher}\\
 X(t,0)&=X(t,1)=0,\quad t>0,\nonumber\\
 X(0,\xi)&=x(\xi),\quad x\in\mathcal{H}\nonumber
\end{align}
$W$ is a cylindrical Wiener process on $\mathcal{H}$, associated to a stochastic basis 
$(\Omega,\mathcal{F},\Pb,(\mathcal{F}_t)_{t\ge 0})$.  $\nu$ is the viscosity coefficient. We will consider the initial condition  
$X(0,\xi)=sech^2(5(\xi-0.5)) $.

We rewrite the Fisher-KPP  equation as an abstract differential equation in $\mathcal{H}$. Set $A=\nu\partial_\xi^2$ and
$B(x)=x(1-x)$, $x\in\mathcal{H}$, with domains $D(A)=H^2(0,1)\cap H_0^1(0,1)$ and $D(B)= H_0^1(0,1)$, respectively. 
Then, \eqref{fisher} can be rewriten as
\begin{align}
 dX&=[AX+B(X)]dt+dW_t\\
 X(0)&=x\quad x\in\mathcal{H}\nonumber.
\end{align}

The operator $A$ is selfadjoint with a complete orthonormal system of eigenfunctions in $\mathcal{H}$ given by
\[
 e_k(\xi)=\sqrt{2}\sin(k\pi\xi), \qquad \xi\in[0,1], k\in\IN.
\]
Moreover $A$ satisfies $Ae_k=-\nu\pi^2k^2e_k$, for $k\in\IN$.

As before we define $ u(t,x)=\E\big[ u_0(X_t^x)\big]$ and then $u(t,x)$ satisfies the Kolmogorov equation
\begin{equation*}
\frac{\partial u}{\partial t}= \frac{1}{2}Tr(QD^2u)+ \langle Ax, Du \rangle_\mathcal{H} + \langle B(x),Du \rangle_\mathcal{H},\qquad x\in D(A).
\end{equation*}
Results on existence and uniqueness of the solution to the Kolmogorov equation can be found, for instance, in \cite[Chapter 4]{da}.

\vspace*{-0.1cm}About the functional $u_0:\mathcal{H}\rightarrow \IR$ we will consider two cases :  
 \begin{align*}
  u_0^{\xi_0}(g)&:= g(\xi_0).\qquad \mbox{for fixed } \xi_0\in(0,1)\\
 &\mbox{\hspace*{-3cm} and }\\
 \vspace*{-0.1cm}  u_0(g)&:=\int_0^1 g(\xi) d\xi.
 \end{align*}

We now apply the numerical method. We write the solution $u$ as 
\begin{align*}
 u(t,x)=\sum_{\bm{n}} u_{\bm{n}}(t) H_{\bm{n}}(x).
\end{align*}
and by following the procedure done before we arrive to an infinite system of ordinary differential equations: 
\begin{equation}
   \dot{u}_{\bm{m}}(t) =   -u_{\bm{m}}(t) \lambda_{\bm{m}} 
  + \sum_{\bm{n}\in \mathcal{J}} u_{\bm{n}}(t) C_{\bm{n},\bm{m}},\qquad \bm{n},\bm{m}\in\mathcal{J}\label{inf-sys-fisher}
\end{equation}
where $ C_{\bm{n},\bm{m}}$ is given by
\[
C_{\bm{n},\bm{m}}=\int_\mathcal{H}
\big\langle B(x), D_x H_{\bm{n}}(x) \big\rangle_\mathcal{H} H_{\bm{m}}(x) \mu(dx) 
\]

we need to calculate the value of the constants $C_{\bm{n},\bm{m}} $, then we need to calculate expressions 
such as $ B(x), D_x H_{\bm{n}}(x)$. 

Focus on the term $B(x)=x(1-x)$. By writing $x=\sum_k\beta_ke_k$, with $\beta_k:= \langle x,e_k \rangle_\mathcal{H}$ 
we have
\begin{align*}
B(x)&= \Big(\sum_k\beta_ke_k \Big) \Big(1-\sum_k\beta_ke_k \Big)=\sum_k\beta_ke_k- \sum_k\sum_l\beta_l\beta_ke_le_k 
\end{align*}

For the expression $D_x H_{\bm{n}}(x)$ we have 
\begin{align*}
 D_x H_{\bm{n}}(x)=\sum_{j=1}^\infty \prod_{\stackrel{i=1}{i\ne j}}^\infty P_{n_i}\big(\langle x,\Lambda^{-1/2}e_i \rangle_{\mathcal{H}}  \big)
 P_{n_j}'\big(\langle x,\Lambda^{-1/2}e_j \rangle_{\mathcal{H}}  \big) \Lambda^{-1/2}e_j 
\end{align*}
Setting $\Lambda = (-A)^{-1}$ and by recalling that  $Ae_j=-\nu\pi^2j^2e_j$ we have $ \Lambda^{-1/2}e_j =
 \sqrt{2\nu}\pi|j| e_j$, 
and by using the last expression we have, 
\begin{align*}
C_{\bm{n},\bm{m}}&= \int_{\mathcal{H}} H_{\bm{m}}(x)\mu(dx) \sum_{j=1}^\infty \prod_{\stackrel{i=1}{i\ne j}}^\infty P_{n_i}
\big(\langle x,\Lambda^{-1/2}e_i \rangle_{\mathcal{H}}  \big)
 P_{n_j}'\big(\langle x,\Lambda^{-1/2}e_j \rangle_{\mathcal{H}}  \big) \sqrt{2\nu}\pi|j| \\
 &\qquad \times \Big[\sum_k \beta_k \big\langle e_k, e_j\big\rangle_{\mathcal{H}}
 - \sum_l\sum_k \beta_l \beta_k \big\langle e_l e_k, e_j\big\rangle_{\mathcal{H}}\Big]\\
 &= \int_{\mathcal{H}} \mu(dx) \sum_{j=1}^\infty \sqrt{2\nu}\pi|j| P_{m_j}\big(\langle x,\Lambda^{-1/2}e_j \rangle_{\mathcal{H}}  \big)
  P_{n_j}'\big(\langle x,\Lambda^{-1/2}e_j \rangle_{\mathcal{H}}  \big)\\
  &\qquad\times \prod_{\stackrel{i=1}{i\ne j}}^\infty P_{n_i} \big(\langle x,\Lambda^{-1/2}e_i \rangle_{\mathcal{H}}  \big) 
  P_{m_i}\big(\langle x,\Lambda^{-1/2}e_i \rangle_{\mathcal{H}}  \big)
\Big[ \beta_j 
 - \sum_l\sum_k \beta_l \beta_k \big\langle e_l e_k, e_j\big\rangle_{\mathcal{H}}\Big] .
\end{align*}

For $N_1\in\IN$ define as before the set $S_{N_1}=\{\bm{n}_1,\bm{n}_2,\ldots,\bm{n}_{N_1}: \bm{n}_i\in J^{M,N}, i=1,\ldots,N_1 \}$. Moreover, 
for $\bm{n},\bm{m}\in S_{M}$ define
\begin{align}
\bar C_{\bm{n},\bm{m}}&:=  \sum_{j=1}^M  \int_{\IR^{M}} P_{m_j}(\xi_j) P_{n_j}'(\xi_j)\mu(d\xi_j) \nonumber\\
&\qquad\qquad \times\prod_{\stackrel{i=1}{i\ne j}}^M P_{m_i}(\xi_i) P_{n_i}(\xi_i)\frac{\mu(d\xi_i)}{\lambda_i} 
\Big[\beta_j - \sum_{l=1}^M \sum_{k=1}^M \beta_l\beta_k \big\langle e_l e_k, e_j\big\rangle_{\mathcal{H}}\Big] \label{C1fisher}.
\end{align}
and the finite system of ordinary differential equations: 
\begin{equation}
   \dot{u}_{\bm{m}}(t) =   -u_{\bm{m}}(t) \lambda_{\bm{m}} 
  + \sum_{\bm{n}\in S_{M}} u_{\bm{n}}(t) \bar C_{\bm{n},\bm{m}},\qquad \mbox{ for each }\bm{m}\in S_{M} \mbox{ and } \bm{n} \in S_{M}.
  \label{fin-sys-fisher}
\end{equation}
Then \eqref{fin-sys-fisher} approximates to the infinite system of ordinary differential equations \eqref{inf-sys-fisher} when 
$N,M \rightarrow \infty $. We use the system  \eqref{fin-sys-fisher} to approximate the solution of the FPK equation associated 
to the Fisher-KPP equation.

\subsubsection{Deterministic equation associated with the stochastic Fisher-KPP Equation.}\label{det2}
 
 Set 
 \begin{align*}
  y(t,\xi)=\E\big[X_t(\xi)\big]
 \end{align*}
then, $y(t,\xi)$ solves the differential equation 
\begin{align}
 \frac{\partial y}{\partial t} &=  \nu\frac{\partial^2 y}{\partial \xi^2}+ +y(t,\xi)\big[1-y(t,\xi)\big]\label{dif3}\\
 y\big|_{t=0}&= \E(X_0).\nonumber
\end{align}

We solve numerically this equation by using the Matlab library {\it pdepe} and we compare our results by using the spectral method with the one
obtained with the pdepe Matlab library.

\subsubsection*{Results on the simulation }

We have the following graphs of simulations using the proposed method with differents values of $J^{N,M}$, $N=4,5$. We make a comparison
with the solution of the deterministic equation, as was described in subsection \ref{det2}, by using the matlab library {\it pdepe}.

We show the results on the simulation for the evaluation functional. The second graph shows the simulation for the 
second functional. The results were obtained with the coefficient $\nu=0.1$. 
\vspace*{0.1cm}

\begin{figure}[H]
 \includegraphics[width = 1.95in]{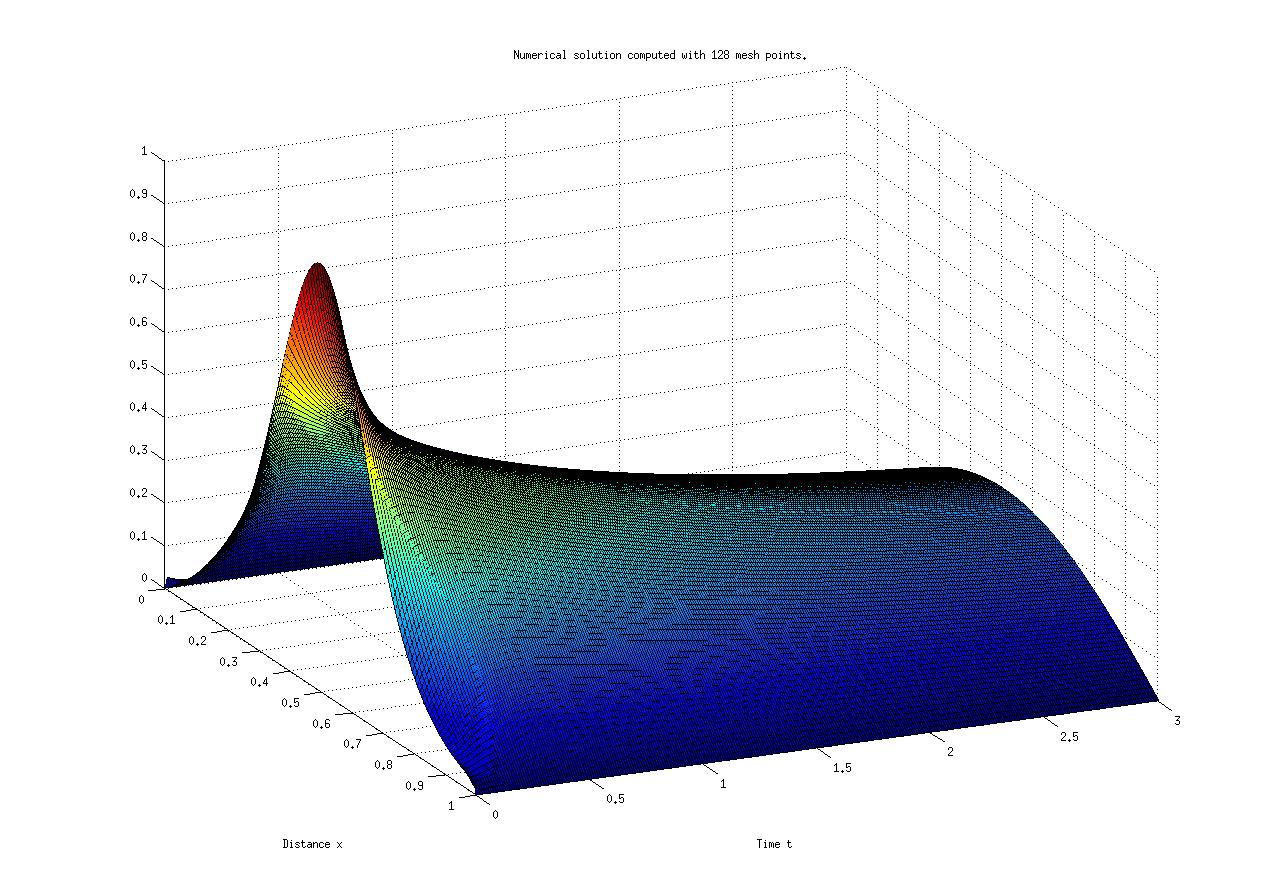}\hspace*{-0.5cm}  
\includegraphics[width = 1.95in]{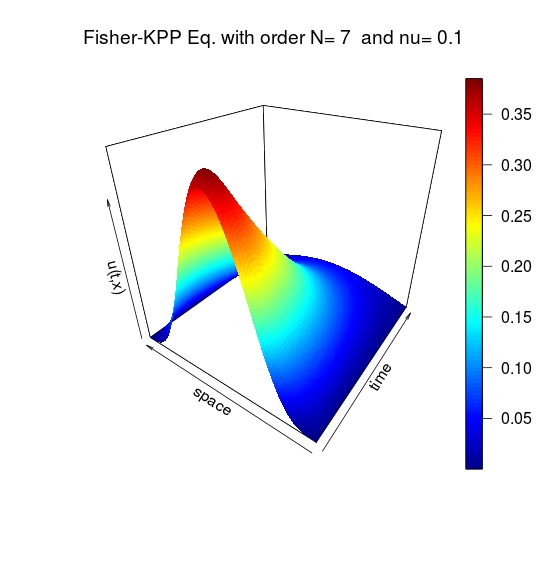}
 \includegraphics[width = 1.95in]{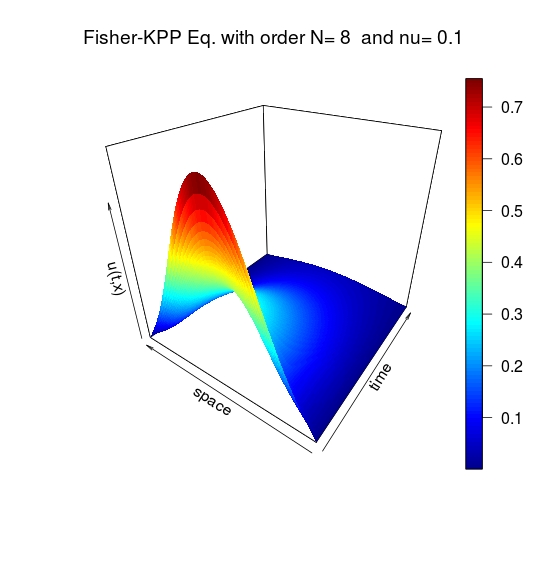} 
 \caption{Simulations for the Fisher-KPP equation with the Matlab library {\it pdepe} and with
 the spectral method for $N=7,8$, $u_0^{\xi_0}(g)=g(\xi_0)$.}
 \label{graph-simu_nu_fisher4.1}
\end{figure}

% 
%  \begin{figure}[H]
% %\includegraphics[scale=0.095]{Diffusion_pdepe_nu_0_1.jpg}
% \includegraphics[width = 1.95in]{fisher-kpp-pdepe_0_01.jpg}\hspace*{-0.5cm}
% \includegraphics[width = 1.95in]{FisherKPP_Eq_N=4_nu_0_01_All.jpeg}
%  \includegraphics[width = 1.95in]{FisherKPP_Eq_N=5_nu_0_01_All.jpeg} 
%  % \includegraphics[width = 1.95in]{FisherKPP_Eq_N=6_nu_0_01_All.jpeg} 
%  \caption{Simulations for the Fisher-KPP equation with the Matlab library {\it pdepe} and with
%  the spectral method for $N=4,5$, $u_0^{\xi_0}(g)=g(\xi_0)$.}
% \label{graph-simu_nu_fisher1.4}
% \end{figure}
% %  

\begin{figure}[H]
 \includegraphics[width = 1.95in]{fisher-kpp-pdepe_0_1.jpg}\hspace*{-0.5cm}  
\includegraphics[width = 1.95in]{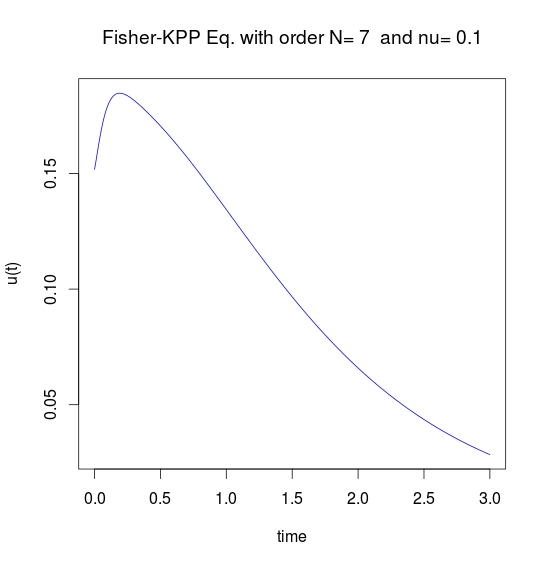}
 \includegraphics[width = 1.95in]{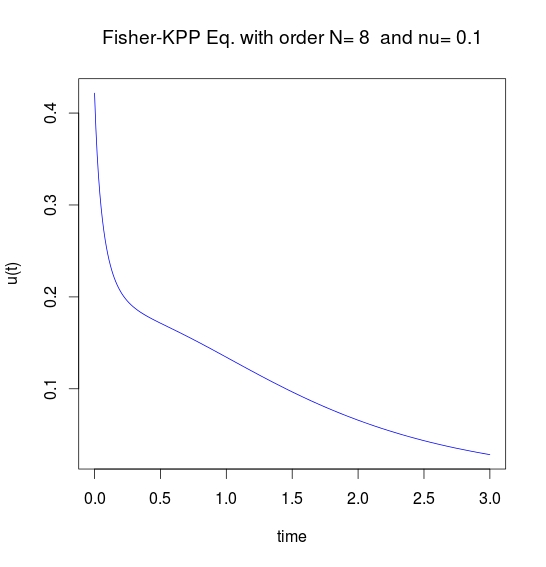} 
 \caption{Simulations for the Fisher-KPP equation with the Matlab library {\it pdepe} and with
 the spectral method for $N=7,8$, $u_0(g)=\int_0^1 g(\xi) d\xi$.}
 \label{graph-simu_nu_fisher5.1}
\end{figure}

% 
% 
%  \begin{figure}[H]
% \includegraphics[scale=0.095]{fisher-kpp-pdepe_0_01.jpg}
% \includegraphics[width = 1.95in]{FisherKPP_Eq_N=4_nu_0_01_u0.jpeg}
% \includegraphics[width = 1.95in]{FisherKPP_Eq_N=5_nu_0_01_u0.jpeg}
% %\includegraphics[width = 1.95in]{FisherKPP_Eq_N=6_nu_0_01_u0.jpeg} 
%  \caption{Simulations for the Fisher-KPP equation with the Matlab library {\it pdepe} and with
%  the spectral method for $N=4,5,6$, $u_0(g)=\int_0^1 g(\xi) d\xi$.}
% \label{graph-simu_nu_fisher1.5}
% \end{figure}

\subsection{Stochastic Burgers Equation in an interval}

Set $\mathcal{H}=L^2(0,1)$. We consider the stochastic Burgers equation in the interval $[0,1]$: 
\begin{align}
 dX(t,\xi)&=\Big[\nu\partial_{\xi}^2X(t,\xi)+\frac{1}{2}\partial_{\xi}(X^2(t,\xi)) \Big]dt+dW_t(t,\xi),\quad t>0,\quad \xi\in(0,1)\label{burg}\\
 X(t,0)&=X(t,1)=0,\quad t>0,\nonumber\\
 X(0,\xi)&=x(\xi),\quad x\in\mathcal{H}
\end{align}
$W$ is a cylindrical Wiener process on $\mathcal{H}$, associated to a stochastic basis $(\Omega,\mathcal{F},\Pb,(\mathcal{F}_t)_{t\ge 0})$. 
$\nu$ is the viscosity coefficient.

We rewrite the Burgers equation as an abstract differential equation in $\mathcal{H}$. Set $A=\nu\partial_\xi^2$ and
$B(x)=\tfrac{1}{2}\partial_\xi(x^2)$, $x\in\mathcal{H}$, with domains $D(A)=H^2(0,1)\cap H_0^1(0,1)$ and $D(B)= H_0^1(0,1)$, respectively. 
Then, \eqref{burg} can be rewriten as
\begin{align}
 dX&=[AX+B(X)]dt+dW_t\\
 X(0)&=x\quad x\in\mathcal{H}\nonumber.
\end{align}

The operator $A$ is selfadjoint with a complete orthonormal system of eigenfunctions in $\mathcal{H}$ given by
\[
 e_k(\xi)=\sqrt{2}\sin(k\pi\xi), \qquad \xi\in[0,1], k\in\IN.
\]
Moreover $A$ satisfies $Ae_k=-\nu\pi^2k^2e_k$, for $k\in\IN$.

As before we define $ u(t,x)=\E\big[ u_0(X_t^x)\big]$ and then $u(t,x)$ satisfies the Kolmogorov equation
\begin{equation*}
\frac{\partial u}{\partial t}= \frac{1}{2}Tr(QD^2u)+ \langle Ax, Du \rangle_\mathcal{H} + \langle B(x),Du \rangle_\mathcal{H},\qquad x\in D(A).
\end{equation*}

Results on existence and uniqueness of the solution to the Kolmogorov equation can be found, for instance, in \cite[Chapter 5]{da}.

\vspace*{-0.1cm} We will consider again two types of functionals :  
 \begin{align*}
  u_0^{\xi_0}(g)&:= g(\xi_0).\qquad \mbox{for fixed } \xi_0\in(0,1)\\
 &\mbox{\hspace*{-3cm} and }\\
 \vspace*{-0.1cm}  u_0(g)&:=\int_0^1 g(\xi) d\xi.
 \end{align*}

We now apply the numerical method. We write the solution $u$ as 
\begin{align*}
 u(t,x)=\sum_{\bm{n}} u_{\bm{n}}(t) H_{\bm{n}}(x).
\end{align*}
and by following the procedure done before we arrive to an infinite system of ordinary differential equations: 
\begin{equation}
   \dot{u}_{\bm{m}}(t) =   -u_{\bm{m}}(t) \lambda_{\bm{m}} 
  + \sum_{\bm{n}\in \mathcal{J}} u_{\bm{n}}(t) C_{\bm{n},\bm{m}},\qquad \bm{n},\bm{m}\in\mathcal{J}\label{inf-sys-burg}
\end{equation}
where $ C_{\bm{n},\bm{m}}$ is given by
\[
C_{\bm{n},\bm{m}}=\int_\mathcal{H}
\big\langle B(x), D_x H_{\bm{n}}(x) \big\rangle_\mathcal{H} H_{\bm{m}}(x) \mu(dx) 
\]

we need to calculate the value of the constants $C_{\bm{n},\bm{m}} $, then we need to calculate expressions 
such as $ B(x), D_x H_{\bm{n}}(x)$. 

Focus on the term $B(x)=\tfrac{1}{2}\partial_\xi(x^2)$. By writing $x=\sum_k\beta_ke_k$, with $\beta_k:= \langle x,e_k \rangle_\mathcal{H}$ 
we have
\begin{align*}
B(x)&=\frac{1}{2}\partial_\xi \Big(\sum_k\beta_ke_k \Big)^2=\frac{1}{2}\partial_\xi\Big[ \sum_l\sum_k\beta_l \beta_k e_l e_k \Big] 
=\frac{1}{2}\sum_l\sum_k\beta_l \beta_k \big(e_l e_k'+e_l' e_k\big).
\end{align*}

For the expression $D_x H_{\bm{n}}(x)$ we have 
\begin{align*}
 D_x H_{\bm{n}}(x)=\sum_{j=1}^\infty \prod_{\stackrel{i=1}{i\ne j}}^\infty P_{n_i}\big(\langle x,\Lambda^{-1/2}e_i \rangle_{\mathcal{H}}  \big)
 P_{n_j}'\big(\langle x,\Lambda^{-1/2}e_j \rangle_{\mathcal{H}}  \big) \Lambda^{-1/2}e_j 
\end{align*}
Setting $\Lambda = (-A)^{-1}$ and by recalling that  $Ae_j=-\nu\pi^2j^2e_j$ we have $ \Lambda^{-1/2}e_j =
 \sqrt{2\nu}\pi|j| e_j$, 
and by using the last expression we have, 
\begin{align*}
C_{\bm{n},\bm{m}}&=\frac{1}{2} \int_{\mathcal{H}} H_{\bm{m}}(x)\mu(dx) \sum_{j=1}^\infty \prod_{\stackrel{i=1}{i\ne j}}^\infty P_{n_i}
\big(\langle x,\Lambda^{-1/2}e_i \rangle_{\mathcal{H}}  \big)
 P_{n_j}'\big(\langle x,\Lambda^{-1/2}e_j \rangle_{\mathcal{H}}  \big) \sqrt{2\nu}\pi|j| \\
 &\qquad \times \sum_l\sum_k \beta_l \beta_k \big\langle e_l e_k'+e_l' e_k, e_j\big\rangle_{\mathcal{H}}\\
 &=\frac{1}{2} \int_{\mathcal{H}} \mu(dx) \sum_{j=1}^\infty \sqrt{2\nu}\pi|j| P_{m_j}\big(\langle x,\Lambda^{-1/2}e_j 
 \rangle_{\mathcal{H}}  \big)
  P_{n_j}'\big(\langle x,\Lambda^{-1/2}e_j \rangle_{\mathcal{H}}  \big)\\
  &\qquad\times \prod_{\stackrel{i=1}{i\ne j}}^\infty P_{n_i} \big(\langle x,\Lambda^{-1/2}e_i \rangle_{\mathcal{H}}  \big) 
  P_{m_i}\big(\langle x,\Lambda^{-1/2}e_i \rangle_{\mathcal{H}}  \big) \sum_l\sum_k \beta_l \beta_k \big\langle e_l e_k'
  +e_l' e_k, e_j\big\rangle_{\mathcal{H}} .
\end{align*}

For $N_1\in\IN$ define as before the set $S_{N_1}=\{\bm{n}_1,\bm{n}_2,\ldots,\bm{n}_{N_1}: \bm{n}_i\in J^{M,N}, i=1,\ldots,N_1 \}$. 
Moreover, for $\bm{n},\bm{m}\in S_{M}$ define
\begin{align}
\bar C_{\bm{n},\bm{m}}&:=\frac{1}{2}  \sum_{j=1}^M \sqrt{2\nu}\pi|j| 
\int_{\IR^{M}} P_{m_j}(\xi_j) P_{n_j}'(\xi_j)\mu(d\xi_j)\nonumber\\
&\qquad\qquad \times\prod_{\stackrel{i=1}{i\ne j}}^M P_{m_i}(\xi_i) P_{n_i}(\xi_i)\mu(d\xi_i)  
\sum_{l=1}^M\sum_{k=1}^M \beta_l \beta_k \big\langle e_l e_k'+e_l' e_k, e_j\big\rangle_{\mathcal{H}}\label{C1burger} .
\end{align}
and the finite system of ordinary differential equations: 
\begin{equation}
   \dot{u}_{\bm{m}}(t) =   -u_{\bm{m}}(t) \lambda_{\bm{m}} 
  + \sum_{\bm{n}\in S_{M}} u_{\bm{n}}(t) \bar C_{\bm{n},\bm{m}},\qquad \mbox{ for each }\bm{m}\in S_{M} \mbox{ and } \bm{n} \in S_{M}.
  \label{fin-sys-burg}
\end{equation}
Then \eqref{fin-sys-burg} approximates to the infinite system of ordinary differential equations \eqref{inf-sys-burg} when 
$N,M \rightarrow \infty $. We use the system  \eqref{fin-sys-burg} to approximate the solution of the FPK equation associated 
with the Burgers equation.

\subsubsection{Deterministic equation associated with the stochastic Burgers Equation.}\label{det3}
 
 Set 
 \begin{align*}
  y(t,\xi)=\E\big[X_t(\xi)\big]
 \end{align*}
then, $y(t,\xi)$ solves the differential equation 
\begin{align}
 \frac{\partial y}{\partial t} &=  \nu\frac{\partial^2 y}{\partial \xi^2}+ +\frac{1}{2}\partial_{\xi}(y^2(t,\xi))\label{dif4}\\
 y\big|_{t=0}&= \E(X_0).\nonumber
\end{align}

We solve numerically this equation by using the Matlab library {\it pdepe} and we compare our results by using the spectral method 
with the one obtained with the pdepe Matlab library.

\subsubsection*{Results on the simulation }

The following graphs show simulations by using the proposed method with differents values of $J^{N,M}$, $N=4,5$. We make a comparison
with the solution of the deterministic equation, as was described in subsection \ref{det3}, by using the matlab library {\it pdepe}.

Tthe results on the simulation for the evaluation functional are in the first group of graphs. The second graph shows the 
simulation for the second functional. The results were obtained with the coefficient $\nu=0.2,0.1,0.01$. 
\vspace*{0.1cm}

 \begin{figure}[H]
\includegraphics[width = 1.95in]{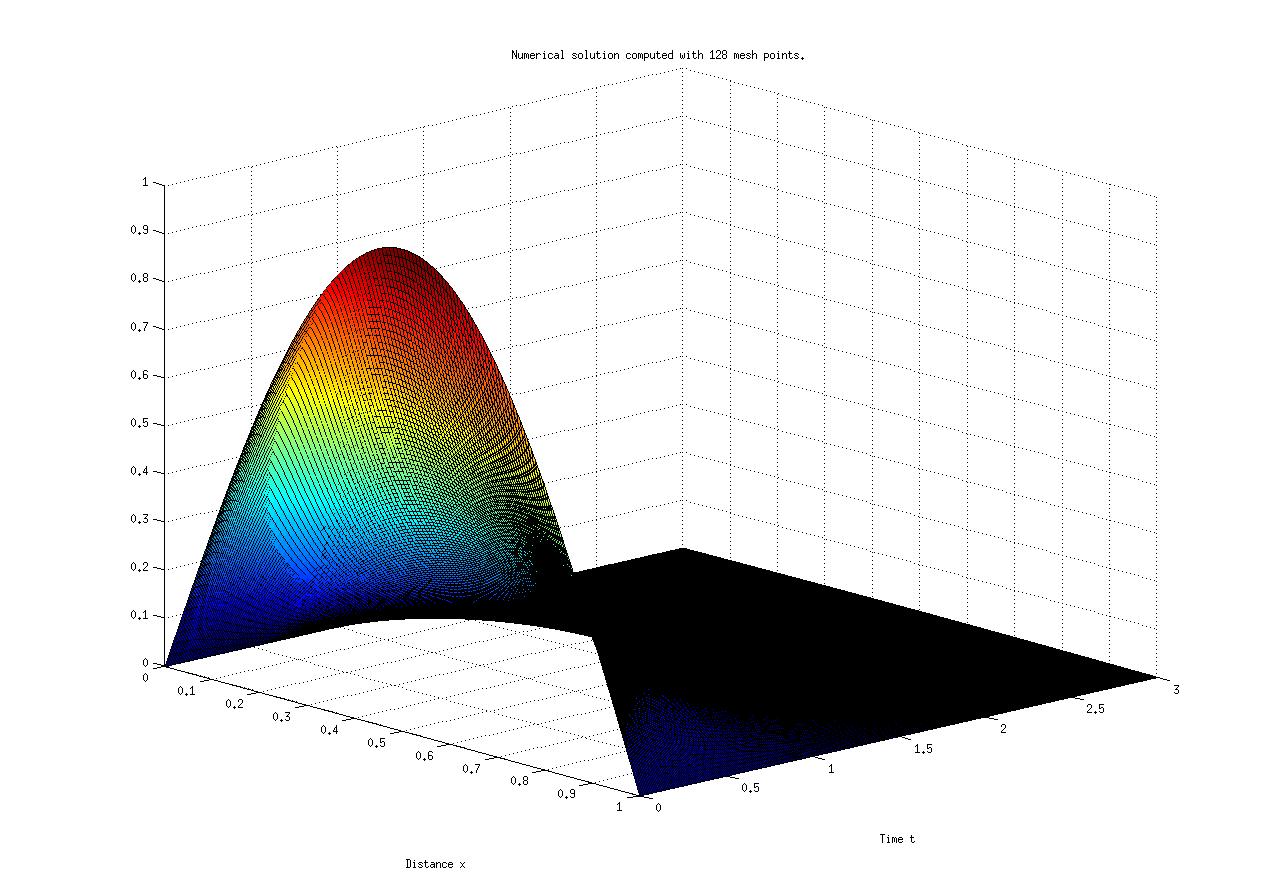}\hspace*{-0.5cm}
\includegraphics[width = 1.95in]{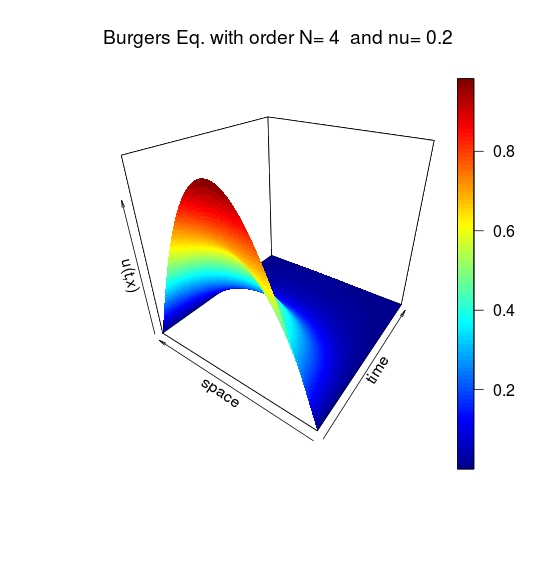}
 \includegraphics[width = 1.95in]{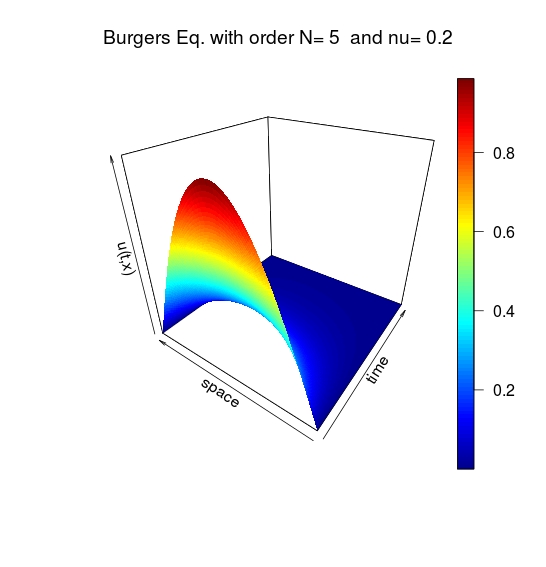} 
 \caption{Simulations for the Burgers equation with the Matlab library {\it pdepe} and with
 the spectral method for $N=4,5$, $u_0^{\xi_0}(g)=g(\xi_0)$.}
\label{graph-simu_nu1.4}
\end{figure}

\begin{figure}[H]
 \includegraphics[width = 1.95in]{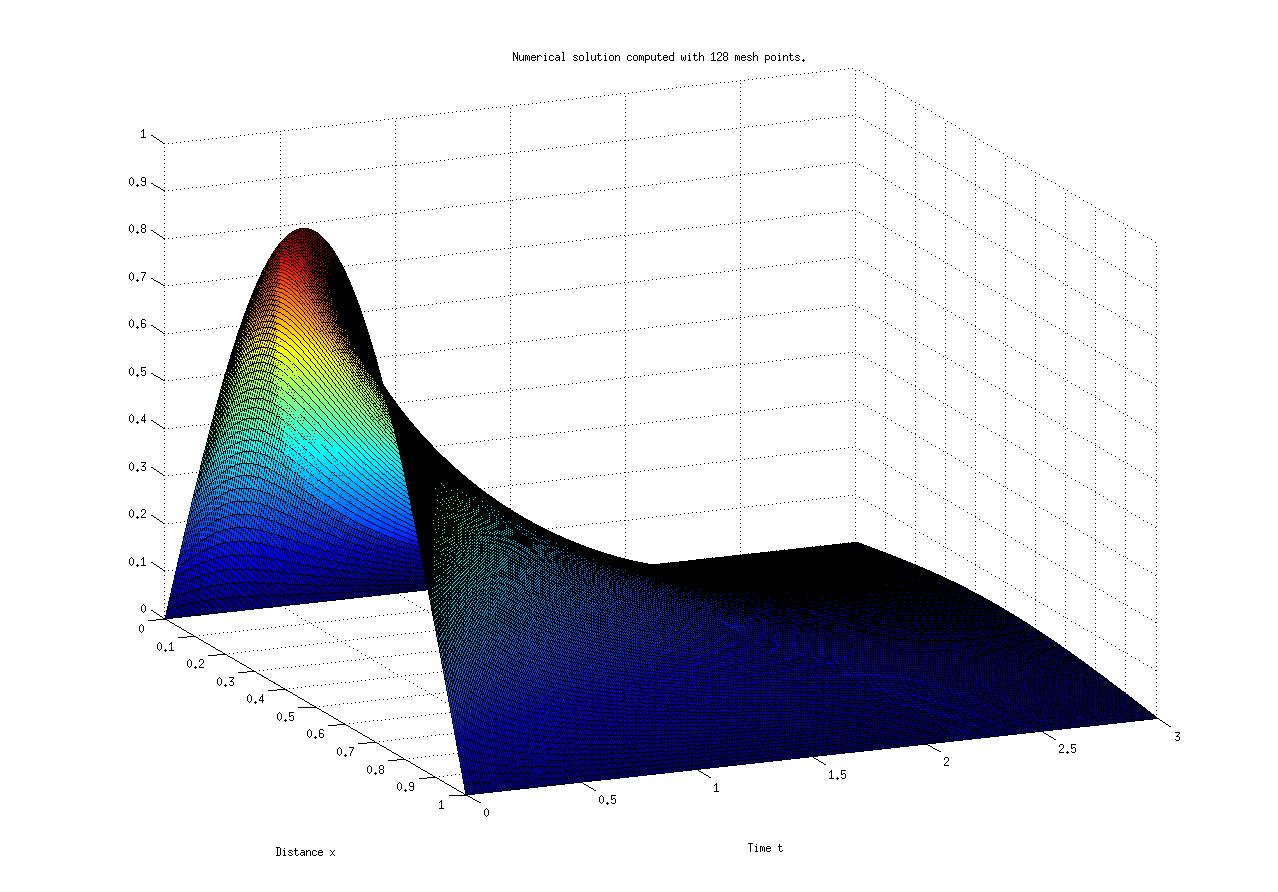}\hspace*{-0.5cm}  
\includegraphics[width = 1.95in]{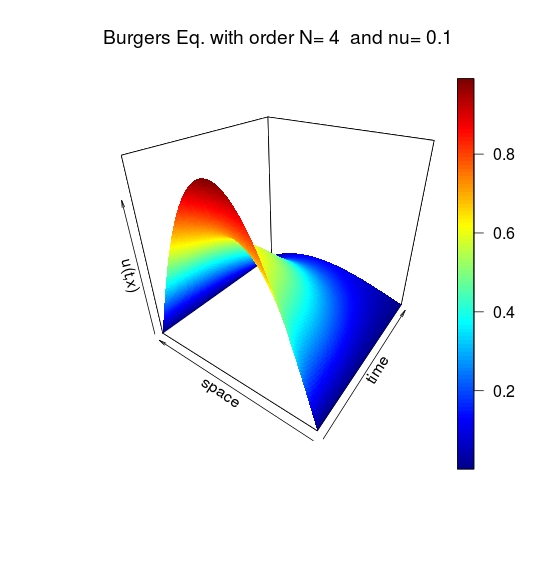}
 \includegraphics[width = 1.95in]{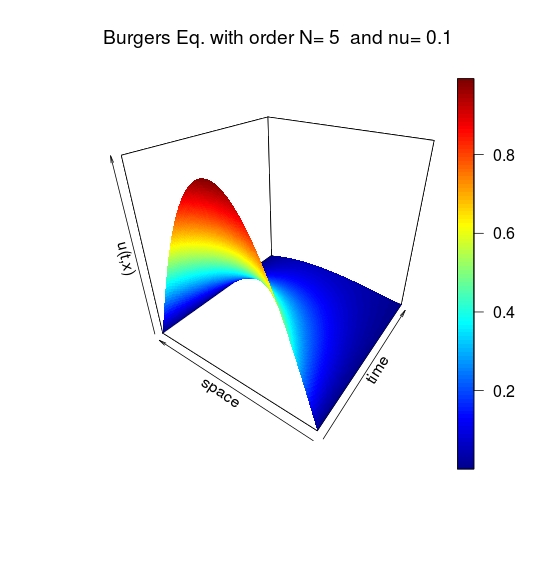} 
 \caption{Simulations for the Burgers equation with the Matlab library {\it pdepe} and with
 the spectral method for $N=4,5$, $u_0^{\xi_0}(g)=g(\xi_0)$.}
 \label{graph-simu_nu4.1}
\end{figure}

 \begin{figure}[H]
  \includegraphics[width = 1.95in]{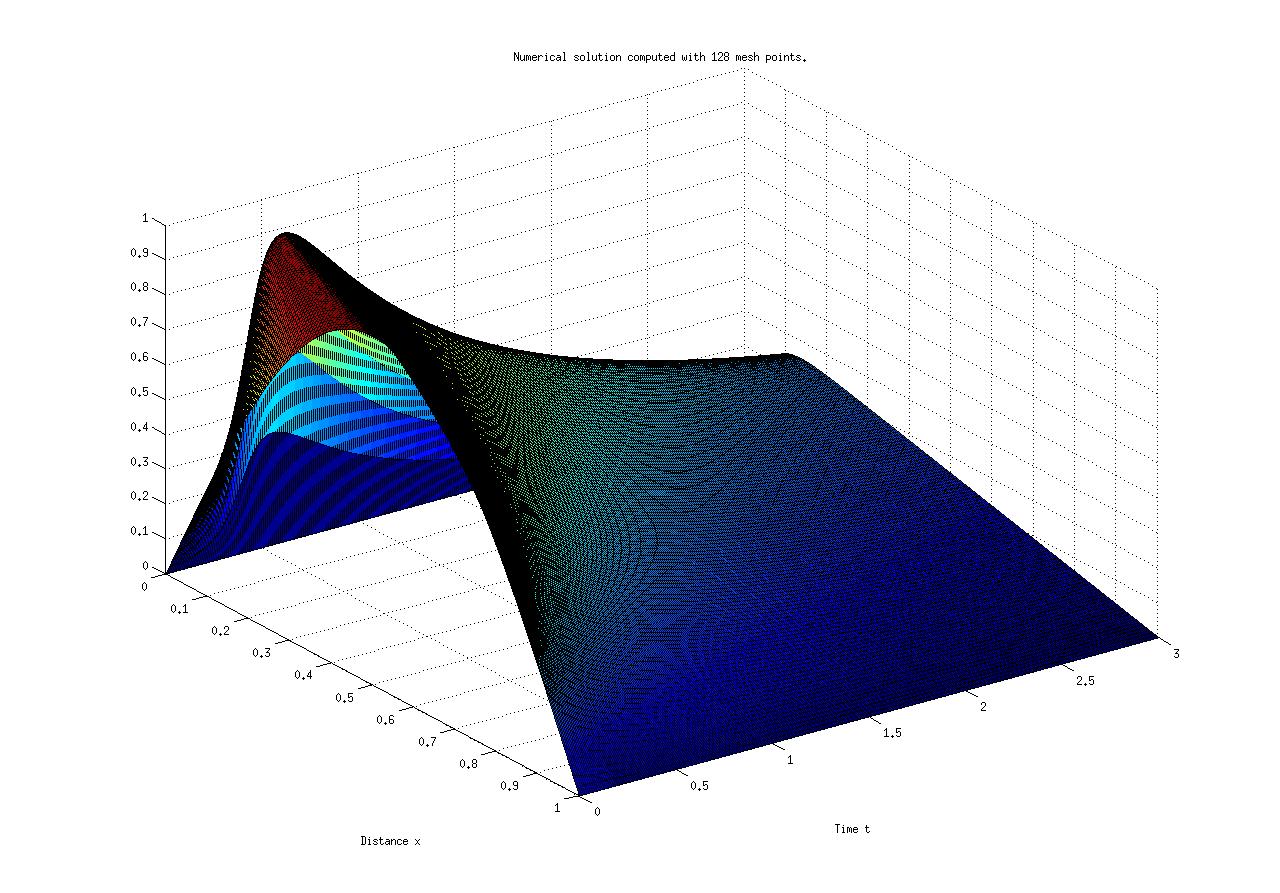}\hspace*{-0.5cm} 
 \includegraphics[width = 1.95in]{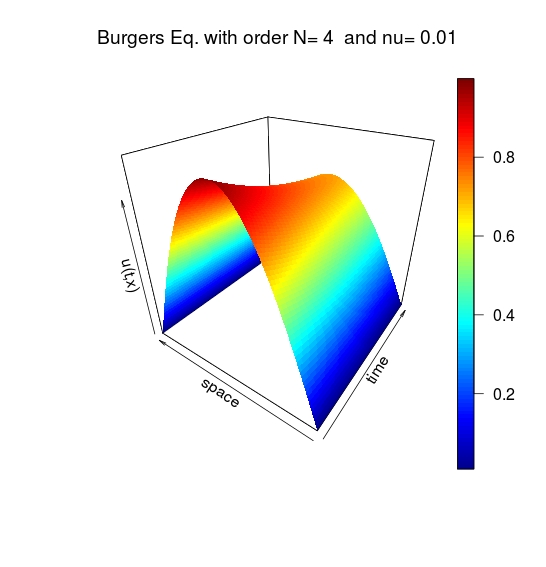}
 \includegraphics[width = 1.95in]{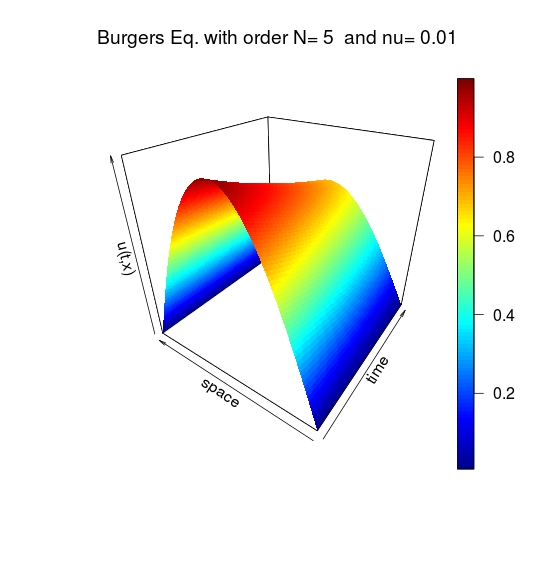} 
 
 \caption{Simulations for the Burgers equation with the Matlab library {\it pdepe}  and with
 the spectral method for $N=4,5$, $u_0^{\xi_0}(g)=g(\xi_0)$.}
 \label{graph-simu_nu2.1}
\end{figure}

 \begin{figure}[H]
\includegraphics[width = 1.95in]{Burgers_PDEPE_nu=0_2.jpg}\hspace*{-0.5cm}
\includegraphics[width = 1.95in]{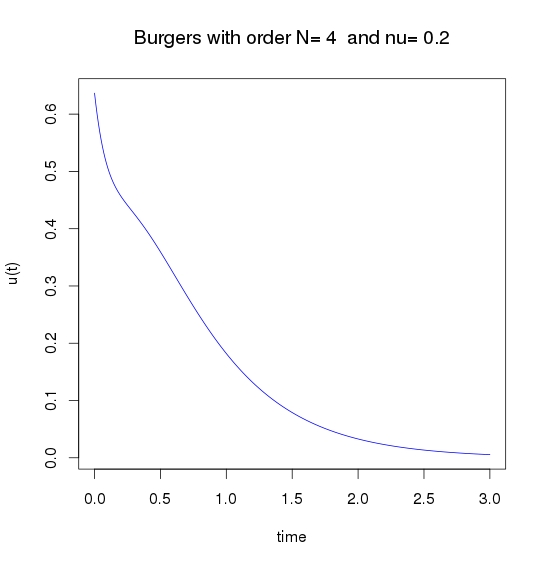}
 \includegraphics[width = 1.95in]{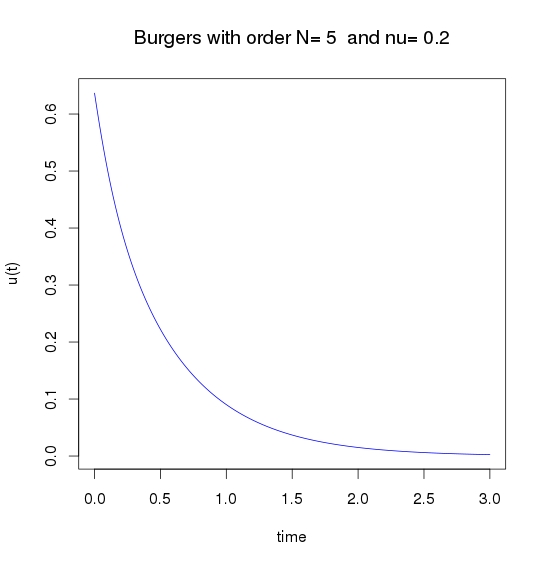} 
 \caption{Simulations for the Burgers equation with the Matlab library {\it pdepe} and with
 the spectral method for $N=4,5$, $u_0(g)=\int_0^1 g(\xi) d\xi$.}
\label{graph-simu_nu1.5}
\end{figure}

\begin{figure}[H]
 \includegraphics[width = 1.95in]{Burgers_PDEPE_nu=0_1.jpg}\hspace*{-0.5cm}  
\includegraphics[width = 1.95in]{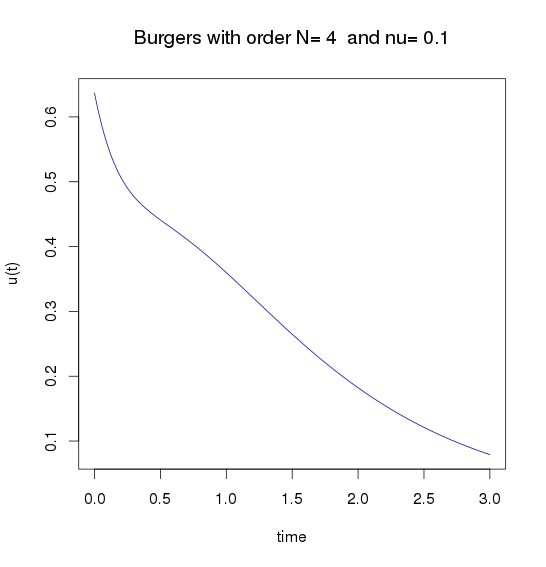}
 \includegraphics[width = 1.95in]{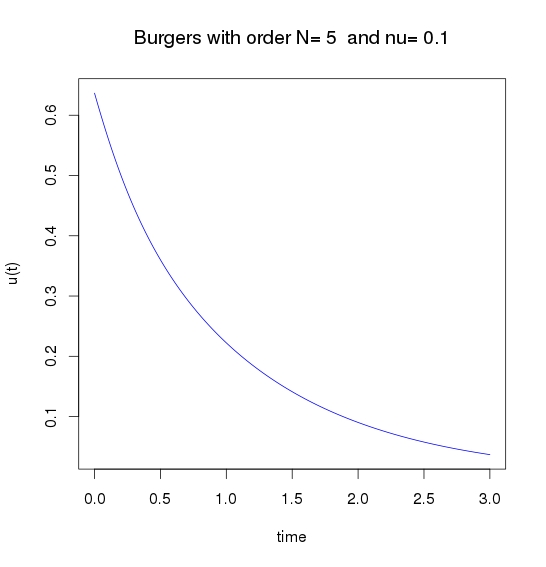} 
 \caption{Simulations for the Burgers equation with the Matlab library {\it pdepe} and with
 the spectral method for $N=4,5$, $u_0(g)=\int_0^1 g(\xi) d\xi$.}
 \label{graph-simu_nu5.1}
\end{figure}

\begin{figure}[H]
  \includegraphics[width = 1.95in]{Burgers_PDEPE_nu=0_01.jpg}\hspace*{-0.5cm} 
 \includegraphics[width = 1.95in]{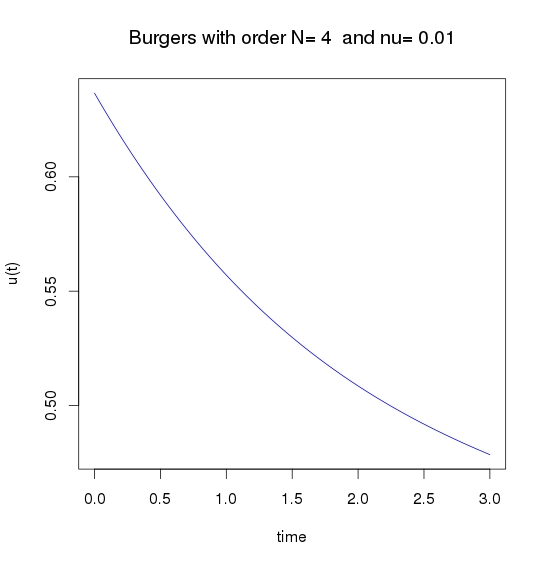}
 \includegraphics[width = 1.95in]{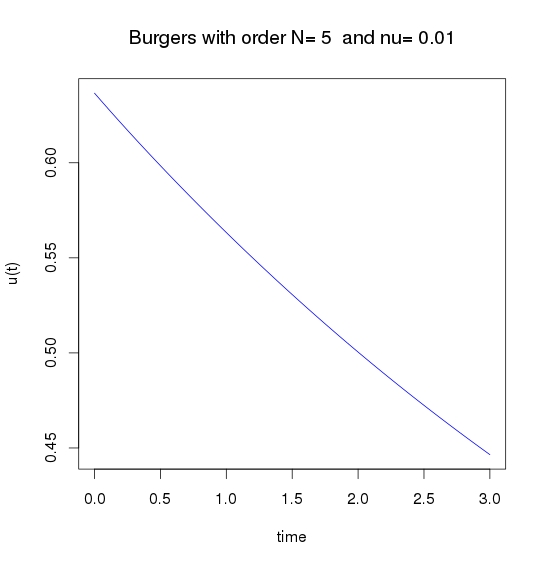} 
 
 \caption{Simulations for the Burgers equation with the Matlab library {\it pdepe}  and with
 the spectral method for $N=4,5$, $u_0(g)=\int_0^1 g(\xi) d\xi$.}
 \label{graph-simu_nu3.1}
\end{figure}

\section{Conclusions.} 
 
In this paper we introduced a numerical method to solve Fokker-Plank-Kolmogorov equations and we tested this
method by applying it to the Kolmogorov equations associated to three stochastic partial differential
equations: a stochastic diffusion, a Fisher-KPP stochastic equation and a stochastic Burgers equation in 1D, in a simple domain in 
the three cases.  The results  obtained  are  really  promising.  However, there  are a few
limitations. The first is that the noise in the SPDE is restricted to the additive case and to cover the
multiplicative case seems unfeasible at this moment. Indeed, even if one is able to prove existence and 
uniqueness of an invariant measure $\nu$ for the Ornstein-Uhlenbeck semigroup associated with the SPDE, there would remain the
fully characterize of the measure and to find a basis for the Hilbert Space $L^2(\mathcal{H},\nu)$. Another issue is that we
have applied the method to very simple domains, However, to cover the cases with complex domains one can use ideas of domain 
decomposition techniques similar to those used in spectral element methods. This is part of a forthcoming paper.

The method can be adapted to cover the Fokker-Plank equations associated with SPDE's, this will be studied in a subsequent work.

\medskip
\noindent{\bf Acknowledgement.} The research leading to these results has received funding from the People Programme
(Marie Curie Actions) of the European Union's Seventh Framework Programme (FP7/2007-2013) under the project NEMOH, 
REA grant agreement n. 289976.


\begin{thebibliography}{99}
\bibitem{ba-da} V. Barbu, G. Da Prato : {\it The Kolmogorov equation for a 2D Navier-Stokes stochastic flow in a channel}. nonlinear
 Analysis. 69, pp. 940-949. (2008).
\bibitem{bo-da-ro} V. Bogachev, G. Da Prato, M. R\"oeckner: {\it Existence results for Fokker-Planck
equations in Hilbert space}.
Seminar on Stochastic Analysis, Random Fields and Applications VI. Progress in Probability Volume 63, 2011, pp 23-35 (2010)
 \bibitem{ca-ma} Cameron R.H., Martin W. T.: {\it The orthogonal development of non-linear functionals in series of Fourier-Hermite functionals}.
 Annals of Mathematics 48 (2): 385-392. (1947).
\bibitem{liu2} P.-L. Chow: {\it Infinte-dimensional Kolmogorov equations in Gauss-Sobolev spaces;}. Stoch.
Analy. Applic. 14, 257-282. (1996).
\bibitem{liu1} P.-L. Chow: {\it Infinite-dimensional parabolic equations in Gauss-Sobolev spaces}. Comm. Stoch. Analy, 
1, 71-86. (2007).
\bibitem{liu} P.-L. chow: Stochastic partial differential equations. Chapman and Hall/CRC. (2007) 
 \bibitem{da} G. Da Prato : Kolmogorov equations for stochastic partial differential equations.
Advanced Courses in Mathematics - CRM Barcelona. Birkh\"auser. (2004).
\bibitem{da-de} G. Da Prato, A. Debussche : {\it $m-$Dissipativity of Kolmogorov Operators Corresponding to Burgers Equations with 
Space-time White Noise}. Potential Analysis. Volume 26, Issue 1 , pp 31-55. (2007). 
\bibitem{da-za} G. Da Prato, J. Zabczyk : Second order partial differential equations in Hilbert spaces.
Cambridge University Press, (2002).
\bibitem{da-za1} G. Da Prato, J. Zabczyk: Stochastic equations in infinite dimensions.
Cambridge University Press, (1992).
\bibitem{da-fl-ro} G. Da Prato, F. Flandoli, M. R\"ockner: {\it Fokker-Planck Equations for SPDE with Non-trace-class Noise}. 
Communications in Mathematics and Statistics. Volume 1, Issue 3, pp 281-304, (2013).
\bibitem{do-ia} Doostan A., Iaccarino G.: "A least-squares approximation of partial differential equations with high 
dimensional random inputs" Journal of Computational Physics, Vol. 228, No. 12, pp. 4332– 4345, 2009.
% \bibitem{du-al-am} A. Dummon, C. Allery, A. Ammar: {\it Proper generalized decomposition method for incompressible
% Navier-Stokes equations with a spectral discretization}. Appl. Math. Comput. 219, no. 15,  
% 8145-8162, (2013).
 \bibitem{gi} M. Giles: {\it Improved multilevel Monte Carlo convergence using the Milstein scheme}, Preprints NA-06/22,
 Oxford University Computing Laboratory, Parks Road, Oxford, U.K., (2006).
 \bibitem{gi1} M. B. Giles: {\it Multilevel Monte Carlo path simulation}, Oper. Res., 56 (2008), pp. 607-617.
 \bibitem{go-sc} A.J. Goldberg, J.L. Schwartz: Systems of Ordinary Differential Equations : An Introduction. 
  Published by Joanna Cotler Books. (1972).
 \bibitem{he} S. Heinrich : {\it Multilevel Monte Carlo Methods} Large Scale Scientific Computing, Third International
Conference, LSSC 2001, Sozopol, Bulgaria, June 6-10, (2001), vol. 2179 of Lecture Notes in Computer Science, Springer, pp. 58-67 (2001).
\bibitem{ho-lu-ro-zh} T. Y. Hou, W. Luo, B. Rozovskii, and H.M. Zhou:{\it Wiener Chaos Expansions and Numerical Solutions of Randomly Forced 
Equations of Fluid Mechanics},   J. Comput. Phys. , 216 , 687-706 (2006). 
% \bibitem{gu-sc-st} R. Guberovic, C. Schwab, R. Stevenson: {\it Space-Time variational saddle point formulations of
% stokes and Navier-Stokes equations}. preprint.
 \bibitem{huo} T. Y. Hou , W. Luo , B. Rozovskii , H.M. Zhou: {\it Wiener chaos expansions and numerical solutions of randomly forced equations of fluid mechanics} 
Journal of Computational Physics archive. Vol 216 Issue 2, (2006) 
% \bibitem{io-la-de} A. Iollo, S. Lanteri, J.-A. Desideri: {\it Stability Properties of POD Galerkin Approximations for the
% Compressible Navier-Stokes Equations}. Theoretical and Computational Fluid Dynamics. Vol. 13,
% Issue 6, pp 377-396, (2000)
\bibitem{im} P. Imkeller: Malliavin's calculus and applications in stochastic control and finance
IMPAN Lecture Notes, Vol. 1, Warsaw (2008).
\bibitem{je-kl} A. Jentzen, P.E. Kloeden: Taylor approximations for stochastic partial differential equations.
CBMS-NSF Regional Conference Series in Applied Mathematics, 83. Society for Industrial and
Applied Mathematics (SIAM), Philadelphia, PA, (2011).
\bibitem{kl-pl} P. E. Kloeden, E. Platen: Numerical Solution of Stochastic Differential Equations. Stochastic 
Modelling and Applied Probability, Vol. 23, Springer (1992).
\bibitem{le-kn} O. P. Le Matre, O.M. Knio: Spectral methods for uncertainty quantification. With applications
to computational fluid dynamics. Scientific Computation. Springer, New York, (2010).
\bibitem {Lions}J. L. Lions: \textit{Quelques methodes de resolution des
problemes aux limites non lineaires}, Dunod, Paris (1969).
\bibitem{lo} S. V. Lototsky. {\it Chaos Approach to Nonlinear Filtering}. In: D. Crisan and B. L. Rozovskii (editors), 
The Oxford Handbook of Nonlinear Filtering, pp. 231-264, Oxford University Press, (2011).
\bibitem{lo-ro} S. V. Lototsky and B. L. Rozovskii. {\it Wiener Chaos Solutions of Linear Stochastic Evolution Equations}.
Annals of Probability, Vol. 34, No. 2, pp. 638-662, (2006)
\bibitem{mi-gu} J. Ming, M. Gunzburger: {\it Efficient numerical methods for stochastic partial differential equations
through transformation to equations driven by correlated noise}. Int. J. Uncertain. Quantif. 3, no. 4, 321-339, (2013).
\bibitem{luo} W. Luo, Wiener Chaos expansion and numerical solutions for stochastic partial differential
equations. PhD thesis. Instituto de Tecnologa de California (2006).
% \bibitem{ma-hu-za} L. Mathelin, Y. Hussaini, and T. A. Zang. {\it Stochastic approaches to uncertainty quantification
% in CFD simulations}. Numerical Algorithms 38.1-3, 209-236, (2005).
% \bibitem{ma-hu-za-1} L. Mathelin, M. Y. Hussaini, T. A. Zang, F. Bataille: {\it Uncertainty propagation for a turbulent,
% compressible nozzle flow using stochastic methods}. AIAA journal, 42(8), 1669-1676, (2004).
\bibitem{pl} E. Platen: {\it An introduction to numerical methods for stochastic differential equations}.
Acta numerica 8, 197-246, (1999).
 \bibitem{sc-su} C. Schwab, E. S\"uli: {\it Adaptive Galerkin approximation algorithms for Kolmogorov equations in infinite dimensions}
 Stochastic Partial Differential Equations: Analysis and Computations. 1:204-239. (2013).
 \end{thebibliography}
\end{document}